\documentclass{scrartcl}

\usepackage[utf8]{luainputenc}
\usepackage[USenglish]{babel}
\usepackage{csquotes}

\usepackage[a4paper,top=27mm,bottom=20mm,inner=25mm,outer=20mm]{geometry}

\usepackage[%
  backend=bibtex,bibencoding=ascii,
  style=numeric-comp,
  giveninits=true, uniquename=init, 
  natbib=true,
  url=true,
  doi=true,
  isbn=false,
  backref=false,
  maxnames=99,
  ]{biblatex}
\usepackage{amsmath}
\allowdisplaybreaks
\usepackage{amssymb}
\usepackage{commath}
\usepackage{mathtools}
\usepackage{bbm}
\usepackage{nicefrac}

\usepackage{siunitx}

\usepackage{amsthm}
\usepackage{thmtools}
\usepackage{etoolbox}
\makeatletter
\patchcmd{\thmt@setheadstyle}
 {\bgroup\thmt@space}
 {\thmt@space}
 {}{}
\patchcmd{\thmt@setheadstyle}
 {\egroup\fi}
 {\fi}
 {}{}
\makeatother
\declaretheoremstyle[
  bodyfont=\normalfont\itshape,
  headformat=\NAME\ \NUMBER\NOTE,
]{myplain}
\declaretheoremstyle[
  headformat=\NAME\ \NUMBER\NOTE,
]{mydefinition}
\newcommand{\envqed}{{\lower-0.3ex\hbox{$\triangleleft$}}}
\declaretheorem[style=myplain,numberwithin=section]{theorem}
\declaretheorem[style=myplain,numberlike=theorem]{lemma}

\declaretheorem[style=mydefinition,numberlike=theorem,qed=\envqed]{remark}
\declaretheorem[style=mydefinition,numberlike=theorem,qed=\envqed]{example}

\usepackage[plainpages=false,pdfpagelabels,hidelinks,unicode]{hyperref}

\usepackage{color}
\usepackage{graphicx}
\usepackage[small]{caption}
\usepackage{subcaption}

\begingroup\expandafter\expandafter\expandafter\endgroup
\expandafter\ifx\csname pdfsuppresswarningpagegroup\endcsname\relax
\else
\fi

\usepackage{booktabs}
\usepackage{rotating}
\usepackage{multirow}

\usepackage{enumitem}

\usepackage{ifluatex}
\ifluatex
  \usepackage[no-math]{fontspec}
\else
  \usepackage[T1]{fontenc}
\fi
\usepackage{newpxtext,newpxmath}

\newcommand{\R}{\mathbb{R}}

\newcommand{\scp}[2]{\left\langle{#1,\, #2}\right\rangle}
\newcommand{\I}{\operatorname{I}}

\DeclarePairedDelimiterX\newset[1]\lbrace\rbrace{\setaux #1||\endsetaux}
\def\setaux#1|#2|#3\endsetaux{\if\relax\detokenize{#2}\relax #1 \else #1 \;\delimsize\vert\; #2 \fi}
\renewcommand{\set}[1]{\newset*{#1}}

\let\epsilon\varepsilon
\let\phi\varphi
\let\rho\varrho

\newcommand{\arXiv}[2]{#2}

\renewcommand{\O}{\mathcal{O}}
\newcommand{\e}{\mathrm{e}}
\newcommand{\1}{\mathbbm{1}}

\newcommand{\dt}{\Delta t}
\newcommand{\dtE}{\dt_\mathrm{FE}}
\newcommand{\dtau}{\Delta \tau}
\newcommand{\ssp}{{\mathcal C}}

\newcommand{\etaest}{\mathcal{H}}

\newcommand{\phiold}{\phi^{\mathrm{old}}}
\newcommand{\told}{t^{\mathrm{old}}}
\newcommand{\uold}{u^{\mathrm{old}}}
\newcommand{\etaold}{\eta^{\mathrm{old}}}
\newcommand{\tnew}{t^{\mathrm{new}}}
\newcommand{\unew}{u^{\mathrm{new}}}
\newcommand{\etanew}{\eta^{\mathrm{new}}}
\newcommand{\tgamma}{t^{n}_\gamma}
\newcommand{\ugamma}{u^{n}_\gamma}
\newcommand{\etagamma}{\eta(u^{n}_\gamma)}

\newcommand{\fnum}{f^\mathrm{num}}

\newenvironment{keywords}{\par\textbf{Key words.}}{\par}
\newenvironment{AMS}{\par\textbf{AMS subject classification.}}{\par}

\title{General Relaxation Methods\texorpdfstring{\\}{ }for Initial-Value Problems\texorpdfstring{\\}{ }with Application to Multistep Schemes}
\author{Hendrik Ranocha, Lajos Lóczi, David I. Ketcheson}
\date{October 17, 2020}

\makeatletter
\hypersetup{pdfauthor={\@author}}
\hypersetup{pdftitle={\@title}}
\makeatother

\begin{document}

\maketitle

\begin{abstract}
  Recently, an approach known as relaxation has been developed for preserving
the correct evolution of a functional in the numerical solution of
initial-value problems, using Runge--Kutta methods.
We generalize this approach to multistep methods, including all
general linear methods of order two or higher, and many other classes of schemes.
We prove the existence of a valid relaxation parameter and high-order
accuracy of the resulting method, in the context of general equations, including
but not limited to conservative or dissipative systems.
The theory is illustrated with several numerical examples.

\end{abstract}



\section{Introduction}
\label{sec:introduction}

Consider an initial-value ordinary differential equation (ODE) in a Banach space:
\begin{equation}
\label{eq:ode}
  u'(t)
  =
  f(u(t))
  \quad
  u(0) = u^{0}.
\end{equation}
Here and in the following, we use upper indices for $t$ and $u$
to denote the index of the corresponding time step.
We say the problem \eqref{eq:ode} is \emph{dissipative} with respect
to a smooth functional $\eta$ if
\begin{subequations}
\label{eq:ode-dissiaptive}
\begin{gather}
  \od{}{t} \eta(u(t)) \leq 0\\
\intertext{for all solutions $u$ of \eqref{eq:ode}, i.e.\ if}
  \forall u\colon \quad \eta'(u) f(u) \leq 0.
\end{gather}
\end{subequations}
In the case of equality in \eqref{eq:ode-dissiaptive},
we say the problem is \emph{conservative}. In the numerical solution of
dissipative or conservative problems, it is desirable to enforce the same
property discretely. For a $k$-step method we thus require
\begin{equation}
\label{eq:discrete-dissipation}
  \eta(u^{n}) \le \max\{ \eta(u^{n-1}), \eta(u^{n-2}), \dots, \eta(u^{n-k}) \}
\end{equation}
for dissipative problems, or
\begin{equation}
\label{eq:discrete-conservation}
  \eta(u^{n}) = \eta(u^{n-1})
\end{equation}
for conservative problems. A numerical method satisfying this requirement is
also said to be dissipative (also known as monotone) or conservative, respectively.

For instance, initial-value problems for hyperbolic or parabolic partial
differential equations (PDEs) usually
have a conserved or dissipated quantity, but in the presence of boundary
and/or source terms this quantity may sometimes increase. In that case,
energy/entropy estimates are still important and the methods developed in this
article are still applicable.

\subsection{Related Work}
Conservative or dissipative ODEs arise in a variety of applications and various
approaches exist for enforcing these properties discretely; for conservative problems
see e.g.\ \cite{hairer2006geometric}, and for dissipative problems see
e.g.\ \cite{dekker1984stability,gottlieb2011strong} and references therein. Besides classical examples
such as Hamiltonian systems, many hyperbolic or hyperbolic-parabolic PDEs
such as the Euler and Navier--Stokes equations
are equipped with an entropy whose evolution in time is important both physically
and for mathematical and numerical stability estimates \cite[Chapter~5]{dafermos2010hyperbolic}.
While there are many semidiscretely entropy-conservative or -dissipative
numerical methods \cite{tadmor2003entropy,lefloch2002fully,fisher2013high,ranocha2018thesis,ranocha2018comparison,sjogreen2018high,friedrich2018entropyHP,chan2018discretely},
transferring such semidiscrete results to fully
discrete schemes is not easy in general. Proofs of monotonicity for
fully discrete schemes have mainly been limited to semidiscretizations
including certain amounts of dissipation \cite{higueras2005monotonicity,zakerzadeh2016high,ranocha2018stability,jungel2017entropy},
linear equations \cite{tadmor2002semidiscrete,ranocha2018L2stability,sun2017stability,sun2019strong},
or fully implicit time integration schemes
\cite{lefloch2002fully,friedrich2019entropy,boom2015high,nordstrom2019energy,ranocha2019some,burrage1979stability,burrage1980nonlinear}.
For explicit methods and general equations, there are negative experimental
and theoretical results concerning energy/entropy stability
\cite{ranocha2020strong,ranocha2020energy,lozano2018entropy,lozano2019entropy}.

To cope with the limitations of time integration schemes, several methods for
enforcing discrete conservation or dissipation have been proposed.
These include orthogonal projections, for one-step methods
\cite{shampine1986conservation,grimm2005geometric}
\cite[Section~IV.4]{hairer2006geometric} and multistep methods
\cite{eich1993convergence,gear1986maintaining,shampine1999conservation},
as well as more problem-dependent techniques for dissipative ODEs such as
artificial dissipation or filtering
\cite{sun2019enforcing,offner2020analysis,offner2018artificial}.
For one-step methods, there are also extensions to projection methods employing
more general search directions via embedded Runge--Kutta (RK) methods
\cite{calvo2006preservation,calvo2010projection,laburta2015numerical}.
Kojima \cite{kojima2016invariants}
reviewed some related methods and proposed another kind of projection scheme
for conservative systems.

The ideas of relaxation methods can be traced back to
\cite{sanz1982explicit,sanz1983method} and \cite[pp. 265--266]{dekker1984stability}.
A relaxation approach was applied in the first two references to the leapfrog
method, and in the third reference to the fourth-order Runge--Kutta method,
in each case to conserve or dissipate an inner-product norm; see also \cite{calvo2006preservation}.
General relaxation Runge--Kutta methods without order reduction have been
proposed and analyzed recently in
\cite{ketcheson2019relaxation,ranocha2020relaxation,ranocha2020relaxationHamiltonian,ranocha2020fully}.
Herein we further generalize the relaxation approach to multistep methods; we focus
on linear multistep methods but the theoretical results apply to virtually any
conceivable method for \eqref{eq:ode}, including for instance all general
linear methods. In the context of partial differential equations, the
relaxation approach is not even limited to a method-of-lines framework.

\subsection{Outline of the Article}

Firstly, we introduce the general relaxation approach for time integration
methods in Section~\ref{sec:general-relaxation}. The proofs of accuracy
and existence of solutions for the relaxation parameter $\gamma$ are
divided into multiple steps and presented in
Sections~\ref{sec:accuracy-u}--\ref{sec:accuracy-gamma-general}, where
the latter section contains the most general result.
Section~\ref{sec:estimaeta-est-1} shows how to compute useful
estimates for the evolution of $\eta(u)$ for non-conservative problems.
In Section~\ref{sec:fixed-fixed-coefficients}, we study the accuracy of multistep
relaxation methods in the case when the method coefficients are not
adapted to account for the variable step size.
Afterwards, we study stability and accuracy properties of relaxation methods in
Section~\ref{sec:stability-properties} and present numerical results
supporting our analysis in Section~\ref{sec:numerical-results}.
Finally, we summarize our findings and present some directions of future
research in Section~\ref{sec:summary}.
\arXiv{%
An additional analysis of superconvergence results for relaxation methods
is contained in the extended version of this article available on arXiv.org
\cite{ranocha2020general}.}{%
An additional analysis of superconvergence results for relaxation methods
is contained in the appendix.}

\section{A General Relaxation Approach}
\label{sec:general-relaxation}

To describe the general relaxation approach, we first write a $k$-step,
order $p$ (with $p \geq 2$) time integration method for the
ODE \eqref{eq:ode} in the form
\begin{align}
\label{eq:normal-step}
  \unew = \psi(f, \dt, u^{n-1}, u^{n-2}, \dots, u^{n-k}).
\end{align}
Here $\unew \approx u(\tnew)$ is the numerical solution that ordinarily
would be used to continue marching in time, and $\tnew = t^{n-1} + \dt$
is the corresponding time of approximation.
But since $\unew$ might violate a desired dissipativity
\eqref{eq:discrete-dissipation} or conservation
\eqref{eq:discrete-conservation} property, we perform a line search
along the (approximate) secant line connecting $\unew$ and a convex
combination $\uold$ of previous solution values, where
\begin{equation}
  \uold = \sum_{i=0}^{m-1} \nu_i u^{n-m+i},
  \qquad
  \told = \sum_{i=0}^{m-1} \nu_i t^{n-m+i},
\end{equation}
for a fixed $m \geq 1$ with $\nu_i \geq 0$ and $\sum_i \nu_i = 1$,
to find a conservative or dissipative solution:
\begin{subequations}
\begin{align}
  \ugamma = \uold + \gamma (\unew - \uold).
\intertext{As we will show, under quite general assumptions,
there is always a positive value of $\gamma$
that guarantees \eqref{eq:discrete-dissipation} or
\eqref{eq:discrete-conservation} and is very close to unity,
so that $\ugamma$ approximates $u(\tgamma)$, where}
\label{eq:t-n-gamma}
  \tgamma = \told + \gamma (\tnew - \told),
\end{align}
\end{subequations}
to the same order of accuracy as the original approximate solution $\unew$.
We will usually suppress the subscript $\gamma$ unless there is a reason to
emphasize this dependence.
Additionally, the dependence of the relaxation parameter
$\gamma$ on the time step is also not written out explicitly.

We now describe how $\gamma$ is chosen at each step.
Given an invariant $\eta$, $\gamma$ is chosen such that
\begin{equation}
\label{eq:eta-est-conservative}
  \eta(\ugamma) = \etaold,
  \qquad
  \etaold = \sum_{i=0}^{m-1} \nu_i \eta(u^{n-m+i}).
\end{equation}
Obviously, if $\eta$ is an invariant and the previous step values were computed
in a conservative way,
then
\begin{equation}
  \etaold
  =
  \sum_{i=0}^{m-1} \nu_i \eta(u^{n-m+i})
  =
  \sum_{i=0}^{m-1} \nu_i \eta(u^{0})
  =
  \eta(u^{0}).
\end{equation}
If $\eta$ is not an invariant, a suitable estimate
\begin{equation}
\label{eq:eta-est-1}
  \etanew
  \approx
  \eta(u(\tnew))
  =
  \eta(\unew) + \O( \dt^{p+1} )
\end{equation}
has to be obtained first. In particular, this estimate
$\etanew$ should be obtained such that the correct sign
of the discrete rate of change can be guaranteed. Then, $\gamma$ has to be
chosen such that
\begin{equation}
\label{eq:condition-for-gamma}
  \eta(\ugamma)
  =
  \etaest(\tgamma)
  :=
  \etaold + \gamma (\etanew - \etaold).
\end{equation}
Obviously, a suitable choice for invariants is $\etanew = \etaold$.
Hence, this approach is a strict generalization of relaxation methods
for invariants to general functionals $\eta$.

In summary, at each time step the general relaxation algorithm consists
of the following substeps:
\begin{enumerate}
  \item
  Define the values
  \begin{equation}
  \label{eq:old-values}
    \begin{pmatrix}
      \told \\
      \uold \\
      \etaold
    \end{pmatrix}
    =
    \sum_{i=0}^{m-1} \nu_i
    \begin{pmatrix}
      t^{n-m+i} \\
      u^{n-m+i} \\
      \eta(u^{n-m+i})
    \end{pmatrix}
  \end{equation}
  as base points of the secants in time, phase space, and ``entropy'' space.
  These old values are convex combinations,
  i.e.\ $\nu_i \geq 0$ and $\sum_i \nu_i = 1$.

  \item
  Compute new values $\tnew$ and
  $\unew \approx u(\tnew) + \O( \dt^{p+1} )$ using a given
  time integration scheme \eqref{eq:normal-step} and a suitable estimate
  $\etanew = \eta(\unew) + \O( \dt^{p+1} )$.

  \item
  Solve the system
  \begin{equation}
  \label{eq:relaxation}
    \begin{pmatrix}
      \tgamma \\
      \ugamma \\
      \etagamma
    \end{pmatrix}
    =
    \begin{pmatrix}
      \told \\
      \uold \\
      \etaold
    \end{pmatrix}
    + \gamma
      \begin{pmatrix}
        \tnew - \told \\
        \unew - \uold \\
        \etanew - \etaold
      \end{pmatrix}
  \end{equation}
  for $\gamma \approx 1$ and continue the integration
  with $\tgamma, \ugamma$ instead of $\tnew, \unew$.
\end{enumerate}

\begin{remark}
  Solving \eqref{eq:relaxation} means
  inserting the second equation into the third and solving the resulting
  scalar equation \eqref{eq:condition-for-gamma} for $\gamma$.
  Then, $\tgamma$ and $\ugamma$ are determined by the first and second
  equation, respectively.
\end{remark}

\begin{remark}
  Throughout this article, the notation $\O(\cdot)$ refers to the limit
  $\dt \to 0$. As mentioned above, superscripts of $t$ and $u$ denote the
  time step. Inside $\O(\cdot)$, superscripts of $\dt$ and $(\gamma-1)$
  denote exponents.
\end{remark}

\begin{remark}
\label{rem:old-values=dissipative}
  The standard choice of the old values \eqref{eq:old-values} is given by
  $m = 1$ and $\nu_0 = 1$, especially for one-step methods.
  For dissipative problems,
  if $m = 1$ and the starting values satisfy
  $\eta(u^{k-1}) \leq \eta(u^{k-2}) \leq \dots \leq \eta(u^{0})$,
  the relaxation approach described in the following will guarantee the
  slightly stronger inequality
  \begin{equation}
    \eta(u^{n})
    \leq \eta(u^{n-1}) \leq \eta(u^{n-2}) \leq \dots \leq \eta(u^{n-k})
    \leq \dots \leq \eta(u^{0})
  \end{equation}
  instead of \eqref{eq:discrete-dissipation}
  if the estimate $\etanew$ is obtained such that the correct sign
  of the discrete rate of change can be guaranteed.
\end{remark}

\begin{remark}
  One could also consider the case $\told = \tnew$, i.e.\
  $\uold \approx u(\tnew)$.
  In that case, for Runge--Kutta methods the relaxation approach
  reduces to the projection method using an embedded pair studied in
  \cite{calvo2006preservation}, where less accuracy of $\uold$
  is needed and the new time $\tnew$ does not need to be adapted.
  Here, we focus on the case $\told < \tnew$.
\end{remark}

  General projection methods replace the numerical solution $\unew$
  with
  \begin{equation}
  \label{eq:projection}
    u^{n}_\lambda
    =
    \unew + \lambda \Phi
    \approx
    u(t^{n}),
  \end{equation}
  where $\Phi$ is a specified search direction and $\lambda$ is chosen
  such that $\eta(u^{n}_\lambda) = \etanew$.
  Projection methods do not modify the new time $\tnew = t^{n}$.
  The projection method used most often in applications is orthogonal
  projection, where $\Phi$ is chosen to minimize the distance
  $\| \unew - u^{n}_\lambda \|$. Often, simplified Newton iterations
  are used in such projection methods \cite[Section~IV.4]{hairer2006geometric}.

The orthogonal projection and relaxation modifications of time integration
schemes are visualized in Figure~\ref{fig:relaxation-visualization} for
$m = 1$.
For conservative problems, both relaxation and orthogonal projection yield
results on the same level set of the invariant $\eta$.
If $\eta$ is convex and the baseline method is anti-dissipative, the
relaxation approach decreases the actual time step $\gamma \dt$,
i.e.\ $\gamma < 1$.
If the baseline scheme is dissipative, relaxation yields larger effective
time steps $\gamma \dt$, i.e.\ $\gamma > 1$.

For dissipative problems, the corrected numerical solutions of orthogonal
projection and relaxation methods are in general on different level sets
of $\eta$. For both anti-dissipative and dissipative baseline schemes,
the value of $\eta(\ugamma)$ is closer to $\etanew$
than $\eta(\unew)$. This is in accordance with the adaptation of the time
\eqref{eq:t-n-gamma}: while $\etanew$ is an approximation
at time $\tnew$, $\eta(\ugamma)$ is an approximation at $\tgamma$.

\begin{figure}[hbt]
\centering
  \hspace*{\fill}
  \begin{subfigure}[t]{0.44\textwidth}
  \centering
    \includegraphics[width=\textwidth]{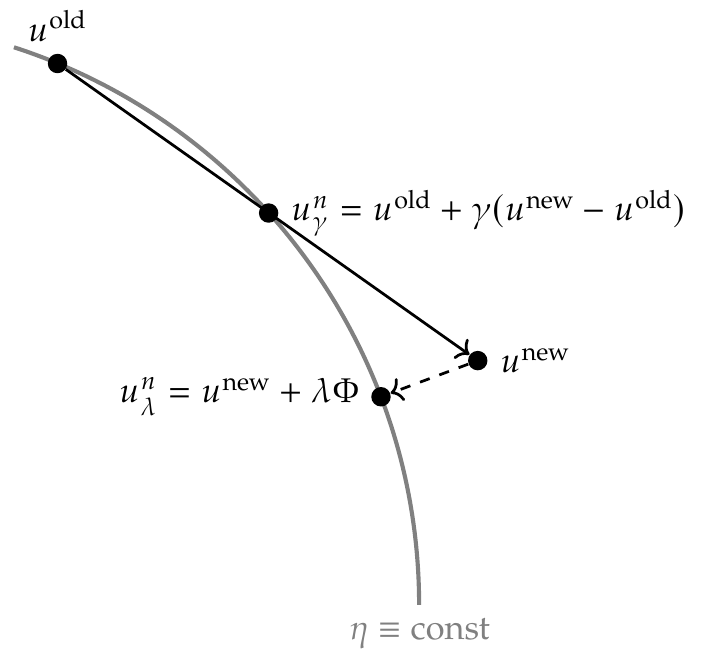}
    \caption{Conservative problem, anti-dissipative method.}
  \end{subfigure}%
  \hspace*{\fill}
  \begin{subfigure}[t]{0.44\textwidth}
  \centering
    \includegraphics[width=\textwidth]{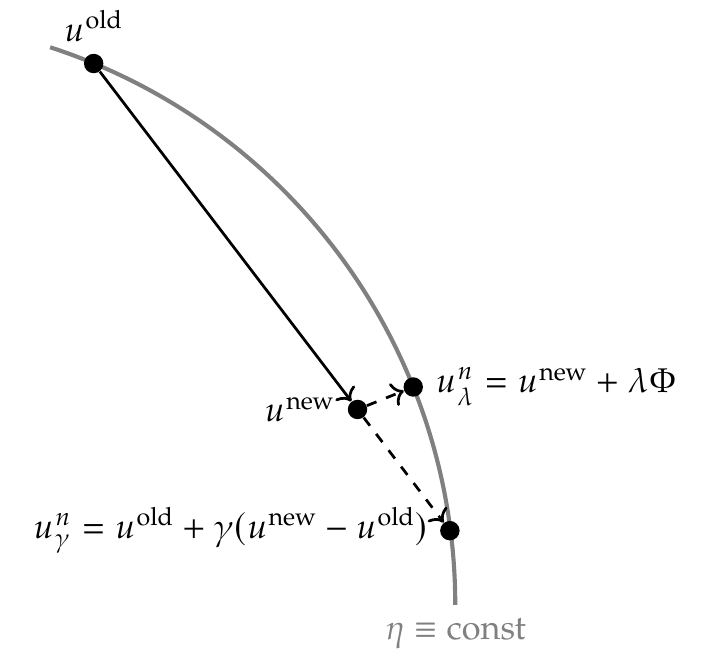}
    \caption{Conservative problem, dissipative method.}
  \end{subfigure}%
  \hspace*{\fill}
  \\
  \hspace*{\fill}
  \begin{subfigure}[t]{0.44\textwidth}
  \centering
    \includegraphics[width=\textwidth]{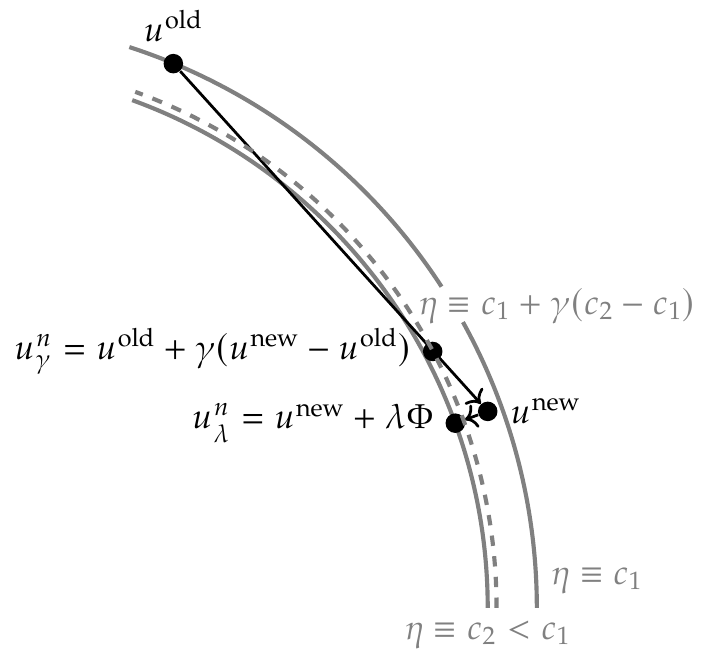}
    \caption{Dissipative problem, anti-dissipative method.}
  \end{subfigure}%
  \hspace*{\fill}
  \begin{subfigure}[t]{0.44\textwidth}
  \centering
    \includegraphics[width=\textwidth]{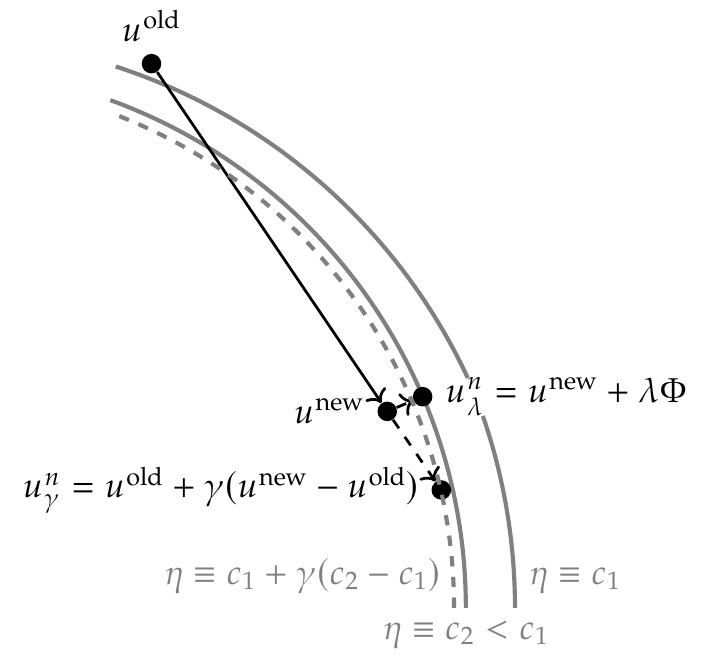}
    \caption{Dissipative problem, dissipative method.}
  \end{subfigure}%
  \hspace*{\fill}
  \caption{Visualization of the orthogonal projection and relaxation
           approaches for conservative/dissipative problems and
           anti-dissipative/dissipative time integration methods
           for $m = 1$ in \eqref{eq:old-values}.}
  \label{fig:relaxation-visualization}
\end{figure}

\begin{remark}
\label{rem:alternative-dissipative-relaxation}
  Based on the sketches shown in Figure~\ref{fig:relaxation-visualization},
  one can expect that it is also possible to choose $\widetilde\gamma$
  such that $\eta(u^{n}_{\widetilde\gamma}) = \etanew$
  for dissipative problems.
  At least for convex problems visualized there, there will be two solutions
  $\widetilde\gamma$ for sufficiently small time steps $\dt$.
  One of these solutions is near unity and the other one is closer to zero.

  In order to solve $\eta(u^{n}_{\widetilde\gamma}) = \etanew$,
  $\widetilde\gamma$ must deviate more from unity than a solution $\gamma$
  of $\eta(\ugamma) = \etaest(\tgamma)$, cf.\ \eqref{eq:condition-for-gamma}.
  Since $\gamma = 1 + \O( \dt^{p-1} )$ for a $p$th order baseline scheme
  as will be shown below, $\widetilde\gamma$ cannot be closer to unity
  than $\O( \dt^{p-1} )$. In the following, we will prove
  \begin{equation}
    \ugamma = u(\tnew) + \O( \dt^{p-1} )
    \quad \text{and} \quad
    \ugamma = u(\tgamma) + \O( \dt^{p+1} ).
  \end{equation}
  Hence, $\etanew$ is only an $\O( \dt^{p-1} )$
  approximation at $\tgamma$ and cannot be better at
  $t^{n}_{\widetilde\gamma}$.
  Thus, solving $\eta(u^{n}_{\widetilde\gamma}) = \etanew$
  will lead to some order reduction and is not pursued further.
\end{remark}

Following the development of the relaxation approach for Runge--Kutta
methods \cite{ketcheson2019relaxation,ranocha2020relaxation}, the
accuracy and suitability of general relaxation time integration methods
is studied in three steps. Firstly, accuracy of the relaxed approximation
\eqref{eq:relaxation} is studied given assumptions on the relaxation parameter
$\gamma \approx 1$. Secondly, existence and accuracy of a suitable
relaxation parameter $\gamma$ satisfying \eqref{eq:condition-for-gamma}
is studied at first for convex $\eta$ and then for general $\eta$.
Finally, methods for computing $\etanew$ are given.

\subsection{Accuracy of the Solution \texorpdfstring{$u$}{u} for Relaxation Methods}
\label{sec:accuracy-u}

Before proving the accuracy of $\ugamma$, we introduce the following
\begin{lemma}
\label{lem:old-values}
  For a smooth function $\phi$, define
  \begin{equation}
    \phiold = \sum_{i=0}^{m-1} \nu_i \phi(u^{n-m+i}),
  \end{equation}
  where $\nu_i$ and $m$ are as in \eqref{eq:old-values}.
  Then, $\phiold = \phi(\uold) + \O( \dt^2 )$.
\end{lemma}
\begin{proof}
  Consider the expansions
  \begin{equation}
    \phiold
    =
    \sum_{i=0}^{m-1} \nu_i \phi(u^{n-m+i})
    =
    \sum_{i=0}^{m-1} \nu_i \left(
      \phi(u^{n-m}) + \phi'(u^{n-m}) (u^{n-m+i} - u^{n-m})
    \right) + \O( \dt^2 )
  \end{equation}
  and
  \begin{equation}
    \phi(\uold)
    =
    \phi\biggl( \sum_{i=0}^{m-1} \nu_i u^{n-m+i} \biggr)
    =
    \phi(u^{n-m})
    + \phi'(u^{n-m}) \biggl( \sum_{i=0}^{m-1} \nu_i u^{n-m+i} - u^{n-m} \biggr)
    + \O( \dt^2 ).
  \end{equation}
  Because of $\sum_i \nu_i = 1$, we have
  $\phiold - \phi(\uold) = \O( \dt^2 )$.
\end{proof}
Note that in general
\begin{equation}
\begin{aligned}
  \uold
  =
  \sum_{i=0}^{m-1} \nu_i u^{n-m+i}
  &=
  \sum_{i=0}^{m-1} \nu_i u(t^{n-m+i}) + \O( \dt^{p+1} )
  \\
  &\neq
  u\biggl( \sum_{i=0}^{m-1} \nu_i t^{n-m+i} \biggr) + \O( \dt^{p+1} )
  =
  u(\told) + \O( \dt^{p+1} ),
\end{aligned}
\end{equation}
except for $m = 1$ or similarly special choices of $\nu = (\nu_i)_i$.
In general, $\uold = u(\told) + \O( \dt^2 )$.

\begin{lemma}
\label{lem:accuracy-u}
  If the method \eqref{eq:normal-step} is of order $p \geq 2$
  and $\gamma = 1 + \O( \dt^{p-1} )$, the relaxation
  solution \eqref{eq:relaxation} is also of order $p$.
\end{lemma}
\begin{proof}
  Use the accuracy of the baseline method \eqref{eq:normal-step}
  and apply Lemma~\ref{lem:old-values} to get the expansion
  \begin{equation}
  \begin{aligned}
    \ugamma
    &=
    \unew + (\gamma - 1) ( \unew - \uold )
    \\
    &=
    u(\tnew)
    + (\gamma - 1) \bigl( u(\tnew) - u(\told) \bigr)
    + \O( \dt^{p+1} )
    + \O\bigl( (\gamma - 1) \dt^2 \bigr)
    \\
    &=
    u(\tnew)
    + u'(\tnew) (\gamma - 1) (\tnew - \told)
    + \O( \dt^{p+1} )
    + \O\bigl( (\gamma - 1) \dt^2 \bigr).
  \end{aligned}
  \end{equation}
  Subtracting the Taylor expansion
  \begin{equation}
  \begin{aligned}
    u(\tgamma)
    &=
    u(\tnew + (\gamma-1) (\tnew - \told))
    \\
    &=
    u(\tnew)
    + u'(\tnew) (\gamma - 1) (\tnew - \told)
    + \O\bigl( (\gamma-1)^2 \dt^2 \bigr)
  \end{aligned}
  \end{equation}
  results in the estimate
  \begin{equation}
    \ugamma - u(\tgamma)
    =
    \O( \dt^{p+1} )
    + \O\bigl( (\gamma-1)^2 \dt^2 \bigr)
    + \O\bigl( (\gamma - 1) \dt^2 \bigr)
    =
    \O( \dt^{p+1} ).
  \end{equation}
\end{proof}

\begin{remark}
  This basic argument (for $m = 1$, i.e.\ $\uold = u^{n-1}$) proves also the accuracy
  of relaxation Runge--Kutta methods if $\gamma = 1 + \O( \dt^{p-1} )$;
  cf.\ \cite{ketcheson2019relaxation,ranocha2020relaxation}.
  In particular, it simplifies the proof of
  \cite[Theorem~2.7]{ketcheson2019relaxation} by avoiding the usual
  Runge--Kutta order conditions.
  In addition, the result holds also for more general schemes such
  as the class of general linear methods \cite{butcher2016numerical}
  or (modified) Patankar--Runge--Kutta methods \cite{burchard2003high}.
\end{remark}

\subsection{Existence and Accuracy of \texorpdfstring{$\gamma$}{γ} for Relaxation Methods for Convex Entropies}
\label{sec:accuracy-gamma-convex}

\begin{lemma}
\label{lem:accuracy-gamma-convex}
  Consider a relaxation method \eqref{eq:relaxation} based on a
  time integration method of order $p \geq 2$.

  If $\eta$ is a convex entropy for the ODE \eqref{eq:ode}, $\dt$ is sufficiently
  small, and $\eta''(\uold)\bigl( f(\uold), f(\uold) \bigr) \neq 0$,
  then there is a unique $\gamma > 0$ that solves
  \eqref{eq:relaxation}.
  This $\gamma$ satisfies $\gamma = 1 + \O( \dt^{p-1} )$.
\end{lemma}
\begin{proof}
  The function $r$ defined in a neighborhood of $[0, 1]$ by
  \begin{equation}
  \label{eq:r-gamma}
  \begin{aligned}
    r(\gamma)
    &=
    \eta(\ugamma) - \etaest(\tgamma)
    \\
    &=
    \eta\bigl( \uold + \gamma ( \unew - \uold ) \bigr)
    - \etaold - \gamma \bigl( \etanew - \etaold \bigr)
  \end{aligned}
  \end{equation}
  is convex, since $\eta$ is convex. Moreover,
  \begin{equation}
    r(0)
    =
    \eta(\uold) - \etaold
    =
    \eta\biggl( \sum_{i=0}^{m-1} \nu_i u^{n-m+i} \biggr)
    - \sum_{i=0}^{m-1} \nu_i \eta(u^{n-m+i})
    \leq 0,
  \end{equation}
  since $\uold$ is a convex combination and $\eta$ is convex. Furthermore,
  \begin{equation}
  \begin{aligned}
    r'(0)
    &=
    \eta'(\uold) ( \unew - \uold )
    - \etanew + \etaold
    \\
    &=
    \etaold - \eta(\unew)
    + \eta'(\uold) ( \unew - \uold )
    + \O( \dt^{p+1} ).
  \end{aligned}
  \end{equation}
  Because of
  \begin{equation}
    \sum_{i=0}^{m-1} \nu_i \eta'(u^{n-m+i})
    =
    \eta'(\uold) + \O( \dt^2 ),
  \end{equation}
  cf.\ Lemma~\ref{lem:old-values}, we have
  \begin{equation}
  \begin{aligned}
    r'(0)
    &=
    - \frac{1}{2} \eta''(\uold)(\unew - \uold, \unew - \uold)
    + \O( \dt^{3} )
    \\
    &=
    - \frac{1}{2} (\tnew - \told)^2 \eta''(\uold)(f(\uold), f(\uold))
    + \O( \dt^{3} )
    < 0
  \end{aligned}
  \end{equation}
  for sufficiently small $\dt > 0$.
  Here, the accuracy of the estimate \eqref{eq:eta-est-1} has been used
  in the second line.
  Similarly,
  \begin{equation}
  \label{eq:r'-at-1-entropy}
  \begin{aligned}
    r'(1)
    &=
    \eta'(\unew) ( \unew - \uold )
    - \etanew + \etaold
    \\
    &=
    \etaold - \eta(\unew)
    - \eta'(\unew) ( \uold - \unew )
    + \O( \dt^{p+1} )
    \\
    &=
    \frac{1}{2} \eta''(\unew)(\uold - \unew, \uold - \unew)
    + \O( \dt^{3} )
    \\
    &=
    \frac{1}{2} (\tnew - \told)^2 \eta''(\unew)(f(\uold), f(\uold))
    + \O( \dt^{3} )
    > 0
  \end{aligned}
  \end{equation}
  for sufficiently small $\dt > 0$.
  Hence, $r$ has a unique positive root $\gamma$.

  Because of the accuracy of the baseline scheme,
  $r(1) = \O( \dt^{p+1} )$.
  Using \eqref{eq:r'-at-1-entropy}, $r'(1) = c \dt^2 + \O( \dt^{3} )$
  with $c > 0$. Hence, the root $\gamma$ of $r$ satisfies
  $\gamma = 1 + \O( \dt^{p-1} )$.
\end{proof}

\begin{remark}
  Instead of the condition $\eta''(\uold)\bigl( f(\uold), f(\uold) \bigr)
  \neq 0$, similar conditions using other step/stage values can be used by
  performing the Taylor expansions around them.
\end{remark}

\begin{remark}
  As can be seen from Figure \ref{fig:relaxation-visualization},
  choosing $\gamma$ slightly smaller than the root of $r(\gamma)=0$
  is a way to introduce additional dissipation.
\end{remark}

\begin{remark}
\label{rem:explicit-gamma}
  If the convex entropy is a squared inner-product norm, i.e.\ if
  $\eta(u) = \frac{1}{2} \|u\|^2$ in a Hilbert space, the relaxation
  parameter can be calculated explicitly as
  \begin{equation}
    \gamma
    =
    \begin{dcases}
      \frac{-b + \sqrt{b^2 - 4 a c}}{2 c}, & a \neq 0, \\
      \frac{-b}{c}, & a = 0,
    \end{dcases}
  \end{equation}
  where
  \begin{equation}
  \begin{gathered}
    a = \eta(\uold) - \etaold \leq 0,
    \qquad
    b = \scp{\uold}{\unew - \uold} - \etanew + \etaold,
    \\
    c = \eta(\unew - \uold) \geq 0.
  \end{gathered}
  \end{equation}
\end{remark}

\begin{theorem}
\label{thm:convex}
  Consider a relaxation method \eqref{eq:relaxation} based on a
  time integration method of order $p \geq 2$.

  If $\eta$ is a convex entropy for the ODE \eqref{eq:ode}, $\dt$ is sufficiently
  small, and $\eta''(\uold)\bigl( f(\uold), f(\uold) \bigr) \neq 0$,
  then there is a unique $\gamma > 0$ that satisfies the relaxation
  condition \eqref{eq:condition-for-gamma} and the resulting relaxation
  method is of order $p$.
\end{theorem}
\begin{proof}
  Combine Lemma~\ref{lem:accuracy-u} and Lemma~\ref{lem:accuracy-gamma-convex}.
\end{proof}

\subsection{Existence and Accuracy of \texorpdfstring{$\gamma$}{γ} for Relaxation Methods for General Functionals}
\label{sec:accuracy-gamma-general}

\begin{theorem}
\label{thm:general}
  Consider a relaxation method \eqref{eq:relaxation} based on a
  time integration method of order $p \geq 2$.
  If $\eta$ is a general (i.e.\ not necessarily convex) smooth
  functional of \eqref{eq:ode}, $\dt$ is sufficiently small, and
  \begin{equation}
    \eta'(\unew) \frac{\unew - \uold}{\| \unew - \uold \|}
    =
    c \dt + \O( \dt^{2} ),
    \quad
    \text{with } c \neq 0,
  \end{equation}
  then there is a unique $\gamma  = 1 + \O( \dt^{p-1} )$ that
  satisfies the relaxation condition \eqref{eq:condition-for-gamma}
  and the resulting relaxation method is of order $p$.
\end{theorem}
\begin{proof}
  Following \cite[Theorem~2]{calvo2010projection}, this proof is based on
  the implicit function theorem, e.g.\ the version of
  \cite[Section~VIII.2]{rosenlicht1986introduction}.

  Given an initial condition $u^{0}$ and a time step $\dt$, approximate solutions
  $u^{n-k+i} \approx u(t^{n-k+i})$ and $\unew \approx u(\tnew)$ are
  computed, e.g.\ via a (relaxation) LMM and a suitable starting procedure.
  In this setting, the proof of \cite[Theorem~2]{calvo2010projection} can
  be adapted, similarly to \cite[Proposition~2.18]{ranocha2020relaxation},
  yielding a unique solution $\gamma  = 1 + \O( \dt^{p-1} )$.
  Because of Lemma~\ref{lem:accuracy-u}, the relaxation method is of order $p$.
\end{proof}

\begin{remark}
  For one-step methods studied in
  \cite{ketcheson2019relaxation,ranocha2020relaxation,ranocha2020relaxationHamiltonian},
  the natural choice of $m,\nu$ in \eqref{eq:old-values} is
  $m = 1$ and $\nu_i = \delta_{i,0}$, where $\delta_{i,0}$ is the
  Kronecker delta. Then, $r(0) = 0$ in \eqref{eq:r-gamma}.
\end{remark}

\section{Suitable Estimates of \texorpdfstring{$\eta$}{η} for Relaxation Methods}
\label{sec:estimaeta-est-1}

For dissipative systems, the relaxation approach requires an
estimate of the entropy at $\tnew$ satisfying
\begin{equation}
\label{eq:entropy-inequality}
  \etanew \le \etaold
\end{equation}
in order to ensure that \eqref{eq:discrete-dissipation} is satisfied
if $m \leq k$.
We first review the approach of \cite{ranocha2020relaxation} for Runge--Kutta methods with
positive weights, showing how it fits naturally into the approach
we have just described. We then discuss how to obtain a suitable
estimate for multistep methods with positive coefficients, and for more
general methods.

\subsection{Runge--Kutta Methods with Positive Quadrature Weights}

A Runge--Kutta method with $s$ stages takes the form
\cite{butcher2016numerical,hairer2008solving}
\begin{subequations}
\label{eq:RK-step}
\begin{align}
\label{eq:RK-stages}
  y^i
  &=
  u^{n-1} + \dt \sum_{j=1}^{s} a_{ij} \, f(t^{n-1} + c_j \dt, y^j),
  \qquad i \in \set{1, \dots, s},
  \\
\label{eq:RK-final}
  \unew
  &=
  u^{n-1} + \dt \sum_{i=1}^{s} b_{i} \, f(t^{n-1} + c_i \dt, y^i).
\end{align}
\end{subequations}

Given the Runge--Kutta stage values, a natural estimate $\etanew$ is
given by using the quadrature rule of the Runge--Kutta method itself:
\begin{equation}
\label{eq:eta-est-1-RK}
  \etanew
  =
  \eta(u^{n-1}) + \dt \sum_{i=1}^s b_i (\eta' f)(y^{i}).
\end{equation}
This corresponds to the approach used in
\cite{dekker1984stability,calvo2006preservation,ketcheson2019relaxation,ranocha2020relaxation}.
The inequality \eqref{eq:entropy-inequality} is guaranteed if the weights $b_i \geq 0$.

\begin{remark}
  The adaptation of the time $\tgamma$ and this choice of the
  estimate appear to be very natural. Indeed, considering the
  augmented system
  \begin{equation}
  \label{eq:ode-t-u-eta}
    \od{}{t}
      \begin{pmatrix}
        t \\
        u(t) \\
        \eta(u(t))
      \end{pmatrix}
      =
      \begin{pmatrix}
        1 \\
        f(u(t)) \\
        (\eta' f)(u(t))
      \end{pmatrix},
  \end{equation}
  the update formula (with $\sum_i b_i = 1$)
  \begin{equation}
    \begin{pmatrix}
        \tgamma \\
        \ugamma \\
        \eta(\ugamma)
      \end{pmatrix}
      =
      \begin{pmatrix}
        t^{n-1} \\
        u^{n-1} \\
        \eta(u^{n-1})
      \end{pmatrix}
      + \gamma \dt \sum_{i=1}^{s} b_{i}
      \begin{pmatrix}
        1 \\
        f(y^i) \\
        (\eta' f)(y^{i})
      \end{pmatrix}
  \end{equation}
  is a natural discretization of \eqref{eq:ode-t-u-eta}, where the
  relaxation parameter $\gamma$ is introduced to enforce the
  consistent evolution of $\eta$.
\end{remark}

\begin{remark}
  In the context of dissipative PDEs such as second-order parabolic ones,
  \eqref{eq:eta-est-1-RK} can provide important (spatial) gradient estimates
  of the stage/step values by bounding $\eta' f$.
\end{remark}

\subsection{Linear Multistep Methods with Positive Coefficients}

A linear multistep method can be written in the form
\begin{equation}
\label{eq:LMM}
  \unew
  =
  \sum_{i = 0}^{k-1} \alpha^n_i u^{n-k+i} + \dt \sum_{i=0}^{k} \beta^n_i f^{n-k+i}.
\end{equation}
The coefficients $\alpha^n_i, \beta^n_i$ have a time index because
they depend on the sequence of step sizes.
If all coefficients $\alpha^n_i, \beta^n_i$ are non-negative,
the method itself can be used to obtain a suitable estimate
$\etanew$.
Indeed, considering again the augmented system
\eqref{eq:ode-t-u-eta}, a high-order estimate is obtained as
\begin{equation}
\label{eq:eta-est-1-SSPLMM}
  \etanew
  =
  \sum_{i \geq 0} \left(
    \alpha^n_i \eta(u^{n-k+i}) + \dt \beta^n_i (\eta' f)(u^{n-k+i})
  \right)
  =
  \eta(\unew) + \O( \dt^{p+1} ).
\end{equation}
Since $\sum_i \alpha^n_i = 1$ for any consistent method
and $\beta^n_i (\eta' f)(u^{n-k+i}) \leq 0$, the dissipation
condition \eqref{eq:entropy-inequality} is guaranteed.
Note in particular that strong stability preserving
(SSP) LMMs have non-negative coefficients \cite[Chapter~8]{gottlieb2011strong}
and can be used in this manner.

\subsection{General Time Integration Methods}

The approach of the previous two subsections relies on non-negativity
of the coefficients $b_i$ and $\alpha_i, \beta_i$, respectively.
For methods with negative coefficients, we can follow the technique
developed in \cite{calvo2010projection}.
In this approach, an estimate $\etanew$
is obtained by interpolation (continuous/dense output) and
a positive quadrature rule. Indeed, consider a quadrature rule of order
at least $p$ with nodes $\tau_i \in [t^{n-m}, \tnew]$ and positive weights
$w_i$, e.g.\ a Gauß quadrature. Compute an interpolant $y(\tau_i)$ of order
at least $p-1$ at the nodes and use
\begin{equation}
\label{eq:eta-est-1-quadrature}
  \etanew
  =
  \eta^{n-m} + \sum_i w_i (\eta' f)(y(\tau_i)).
\end{equation}
Because $w_i > 0$, the estimate is guaranteed to satisfy \eqref{eq:entropy-inequality}
for dissipative systems.

In contrast to the previous method, this approach requires additional
evaluations of $\eta'$ and $f$ at the intermediate values and is thus more
costly. On the other hand, it can be applied to general time integration
methods and does not require any special structure (besides a
continuous output formula). In particular, it is not limited to Runge--Kutta
or linear multistep methods.

\begin{remark}
  This approach is particularly interesting for linear multistep methods,
  since these schemes are often defined naturally in terms of (interpolating)
  polynomials \cite{arevalo2017grid,mohammadi2019polynomial}.
\end{remark}

\section{Relaxation Linear Multistep Methods With Fixed Coefficients}
\label{sec:fixed-fixed-coefficients}

Since orthogonal projection does not inherently involve any change in the step
size, it can be used with a fixed step size.
In contrast, relaxation methods necessarily introduce variation in the step
size, due to the parameter $\gamma$, even if the intended $\dt$ at each step is
constant.
To make use of Theorems~\ref{thm:convex} and \ref{thm:general} to guarantee
high-order accuracy, multistep relaxation methods must be implemented in a way
that takes into account the variation in the step size.
On the other hand, since $\gamma$ is very close to unity, this step-size variation
is quite small.
Thus it is interesting to know what accuracy may be obtained by relaxation
methods using the new time step \eqref{eq:t-n-gamma} but implemented with the coefficients of the
corresponding fixed step size method.
We will write simply $\alpha_i, \beta_i$ (without superscripts) to refer
to the coefficients of a fixed step size LMM.

\begin{theorem}
    Given a $k$-step LMM \eqref{eq:LMM} of order $p\ge 1$ with coefficients $(\alpha,\beta)$,
    suppose that the method is used with step sizes $\dt_{n-k+j} = \gamma_j \dt_n$,
    where $\gamma_j-1=\O(\dt^{q})$ but the coefficients are not adapted. Then
    the one-step error is $\O(\dt^{\min(p+1,q+1)})$.
    Furthermore, for Adams methods the one-step error is $\O(\dt^{\min(p+1,q+2)})$.
\end{theorem}
\begin{proof}
    The one-step error takes the form \cite[p. 133]{rjl:fdmbook},
    \cite[Eqn. (2.16)]{hadjimichael2016strong}
    \begin{align} \label{error-expansion}
        E & = \sum_{\ell=0}^\infty \dt^\ell u^{(\ell)}(\tnew) C_\ell(\alpha,\beta,\gamma)
    \end{align}
    where $C_0=\sum_j\alpha_j - 1$ and for $\ell>0$
    \begin{align} \label{error-coeff}
        C_\ell(\alpha,\beta,\gamma) = \sum_{j=0}^{k-1}\left(\Omega_j^\ell \alpha_j + \ell \Omega_j^{\ell-1}\beta_j\right) - \Omega_k^\ell
    \end{align}
    with
    \begin{align}
        \Omega_j & = \sum_{i=1}^j \frac{\dt_{n-k+i}}{\dt_n} = \sum_{i=1}^j \gamma_j = j + \sum_{i=1}^j (\gamma_j-1) = j + \delta_j,
    \end{align}
    where $\delta_j = \O(\dt^q)$.
    Since the LMM has order $p$, its coefficients satisfy the fixed step size order conditions $C_0=0$ and
    \begin{align} \label{lmm_oc}
        C_\ell^{\text{unif}} : = C_\ell(\alpha,\beta,\1) = \sum_{j=0}^{k-1}\left(j^\ell \alpha_j + \ell j^{\ell-1} \beta_j\right) - k^\ell & = 0, & \ell \in\{1,2,\dots,p\}.
    \end{align}
    Subtracting \eqref{lmm_oc} from \eqref{error-coeff} shows that $C_\ell=\O(\dt^{q})$ for
    $1\le \ell \le p$.
    Substitution into \eqref{error-expansion} gives the result stated in the first part of
    the theorem. For the second result, observe that
    \begin{align}
        C_1(\alpha,\beta,\gamma) & = \sum_{j=0}^{k-1}\left((j+\delta_j)\alpha_j + \beta_j\right) - k - \delta_k \\
                & = C_1^{\text{unif}} + \sum_{j=0}^{k-1}(\delta_j \alpha_j) - \delta_k = \sum_{j=0}^{k-1}(\delta_j \alpha_j) - \delta_k.
    \end{align}
    Since $\gamma_k=1$, we have $\delta_k=\delta_{k-1}$. For Adams
    methods we have $\alpha_{k-1}=1$, while $\alpha_j=0$ for $j\ne k-1$,
    so $C_1(\alpha,\beta,\gamma)=0$.
\end{proof}
Note that in the theorem we have analyzed the accuracy of $\unew$
rather than that of $\ugamma$, but the proof of Lemma~\ref{lem:accuracy-u}
shows that $\ugamma$ will have the same order of accuracy.
According to Lemma~\ref{lem:accuracy-gamma-convex}, the assumption in the
theorem is fulfilled for relaxation methods with $q = p-1$. This suggests that
(for non-Adams methods) the local error in the first step
will be one order worse than the design order of the method. But then
for the next step Lemma~\ref{lem:accuracy-gamma-convex} shows that $\gamma-1$
will also be one order worse. In this manner, the local error can grow by one order
at each step, until very quickly all accuracy is lost and/or a suitable
solution $\gamma > 0$ cannot be found. We see that Adams methods are free from
this problem. A cure for this problem for another classes of LMMs is discussed next.

Consider a linear multistep method \eqref{eq:LMM} with non-negative coefficients $\alpha_i$
and take the special choice $\nu_i = \alpha_i$ for computing the old values in
\eqref{eq:old-values}. This yields
\begin{equation}
  \begin{pmatrix}
    \tnew \\
    \unew
  \end{pmatrix}
  -
  \begin{pmatrix}
    \told \\
    \uold
  \end{pmatrix}
  =
  \dt \sum_{i=0}^k \beta_i
  \begin{pmatrix}
    1 \\
    f^{n-k+i}
  \end{pmatrix}.
\end{equation}
Instead of scaling the time step, this relaxation method can be interpreted
as scaling the right hand side of the augmented ODE \eqref{eq:ode-t-u-eta}
by introducing the pseudotime $\tau$ with constant time steps $\dtau$
and solving
\begin{equation}
\label{eq:ode-t-u-eta-tau}
  \od{}{\tau}
  \begin{pmatrix}
    t(\tau) \\
    u(\tau) \\
    \eta(u(\tau))
  \end{pmatrix}
  =
  \Gamma(\tau)
  \begin{pmatrix}
    1 \\
    f(u(\tau)) \\
    (\eta' f)(u(\tau))
  \end{pmatrix},
  \qquad
  \od{}{\tau} \Gamma(\tau) = \gamma(\tau),
\end{equation}
where $\gamma(\tau)$ is the relaxation parameter. If $\Gamma$ is
continuous and bounded away from zero, this new augmented ODE
\eqref{eq:ode-t-u-eta-tau} results from \eqref{eq:ode-t-u-eta}
by the variable transformation
\begin{equation}
\label{eq:t-tau}
  t(0) = 0,
  \qquad
  \od{t(\tau)}{\tau} = \Gamma(\tau).
\end{equation}
Using this interpretation (and choice of $\nu_i$), relaxation LMMs
with $\alpha_i\ge 0$ can also be used with fixed coefficients.
\begin{theorem}
\label{thm:fixed-fixed-coefficients}
  Consider a fixed coefficient linear multistep method \eqref{eq:LMM} of
  order $p \geq 2$ with non-negative coefficients $\alpha_i \geq 0$
  and the corresponding relaxation method \eqref{eq:relaxation} with
  $\nu_i = \alpha_i$.
  If $\eta$ is a smooth functional of \eqref{eq:ode},
  $\dt$ is sufficiently small,
  and the non-degeneracy condition of
  Theorem~\ref{thm:convex} (if $\eta$ is convex)
  or Theorem~\ref{thm:general} (for general $\eta$) is satisfied,
  then there is a solution $\gamma$ of \eqref{eq:relaxation}
  and the resulting relaxation method is of order $p$.
\end{theorem}
\begin{proof}
  Since the fixed step size method is of order $p \geq 2$,
  Theorem~\ref{thm:convex} (if $\eta$ is convex)
  or Theorem~\ref{thm:general} (for general $\eta$)
  can be applied. Hence, there is a solution
  $\gamma(\tau) = 1 + \O( \dtau^{p-1} )$ of
  \eqref{eq:relaxation}.

  Because of the special choice of $\nu_i = \alpha_i$,
  the sequence of the relaxation solutions is a $p$th order approximation
  to the solution of \eqref{eq:ode-t-u-eta-tau}. Hence, it is also a $p$th
  order approximation of the solution of \eqref{eq:ode}.
\end{proof}

\begin{remark}
\label{rem:fixed-fixed-coefficients}
  Because of the scaling by $\Gamma(\tau)$ in \eqref{eq:ode-t-u-eta-tau},
  the relaxation parameter $\gamma$ may be further than expected from unity.
  Indeed, $\gamma(\tau) = 1 + \O( \dtau^{p-1} )$ yields
  $\Gamma(\tau^\mathrm{final}) = 1 + \O( \dtau^{p-2} )$,
  since $\O( \dtau^{-1} )$ time steps have to be used.
  Hence, the observed alteration of the physical time $t$ is
  $\max_\tau |\Gamma(\tau) \gamma(\tau) - 1| = \O( \dtau^{p-2} )$.
\end{remark}

\begin{remark}
\label{rem:adapted-vs-fixed-coefficients}
  A relaxation LMM \eqref{eq:relaxation} with adapted coefficients
  is typically more accurate than the same relaxation LMM with fixed coefficients,
  in particular for second-order methods.
  Indeed, the factor $\Gamma(\tau)$ in the underlying ODE
  \eqref{eq:ode-t-u-eta-tau} grows as
  $\Gamma(\tau) = 1 + \O( \dtau^{p-2} )$. Hence, it does not
  decrease in size if $\dtau$ is reduced for $p = 2$. Since an increase
  of $\Gamma(\tau)$ results in an increase of the Lipschitz constant of
  the underlying ODE \eqref{eq:ode-t-u-eta-tau}, it is reasonable to expect
  bigger errors compared to a relaxation method with adapted coefficients
  solving \eqref{eq:ode-t-u-eta}.
\end{remark}

\section{Stability and Accuracy Properties of Relaxation Methods}
\label{sec:stability-properties}

Since the relaxation parameter $\gamma = 1 + \O( \dt^{p-1} )$
introduces only a small variation of the baseline time integration scheme,
basic stability properties are often not lost, similarly to the case of
relaxation Runge--Kutta methods \cite[Section~3]{ketcheson2019relaxation}.
Furthermore, applying relaxation can increase the accuracy of baseline
schemes.

The incremental direction technique (IDT) approach is basically the
relaxation approach without adapting the new time to \eqref{eq:t-n-gamma}.
For Runge--Kutta methods, this results in a slight loss of the order of
accuracy \cite{calvo2006preservation,ketcheson2019relaxation,ranocha2020relaxation},
i.e.\ the IDT method is of order $p-1$.
For some LMMs and conservative/dissipative problems, IDT versions result
in $\gamma = 1 + \O( \dt^{p-2} )$ and an order of accuracy $p-2$.
However, there are also dissipative problems where IDT versions fail because
no solution for $\gamma$ can be found.

\subsection{Zero-Stability of Relaxation Linear Multistep Methods}

In general, proving even zero-stability of LMMs with variable step sizes is not
easy \cite{soderlind2018zero}.
Typically, stability of variable step size LMMs is implied if
the corresponding fixed step size method has certain stability
properties and the step sizes do not vary too much, cf.\ e.g.\
\cite[Theorems~III.5.5 and III.5.7]{hairer2008solving} or
\cite{arevalo2017grid}.
For relaxation LMMs, the step sizes are chosen via the relaxation
coefficient $\gamma = 1 + \O( \dt^{p-1} )$. Hence,
the step sizes vary as
\begin{equation}
  \frac{\dt_{n}}{\dt_{n-1}}
  =
  \frac{\gamma_{n} \dt}{\gamma_{n-1} \dt}
  =
  1 + \O( \dt^{p-1} ).
\end{equation}
Thus, under the usual conditions, relaxation LMMs are
stable if the time step is small enough.

\subsection{Strong Stability Preserving Methods}

SSP methods with SSP coefficient $\ssp > 0$ guarantee a given convex stability
property under a time step restriction $\dt \leq \ssp \dtE$
whenever the explicit Euler method satisfies the same convex stability
property under the time step restriction $\dt \leq \dtE$; cf.\
\cite{gottlieb2011strong} and the references cited therein.

If the relaxation parameter $\gamma \in [0, 1]$, then $\ugamma$ is a convex
combination of $\unew$ and $\uold$. Hence, all convex
stability properties satisfied by these two values are retained.
However, if $\gamma > 1$, $\ugamma$ is not a convex combination
and the SSP property with the same SSP coefficient~$\ssp$ can be lost, cf.\
\cite[Theorem~3.3 and Table~1]{ketcheson2019relaxation}, where also more
detailed investigations of the SSP property of relaxation Runge--Kutta
methods can be found. Specifically, therein it was proved that the
SSP coefficient of a relaxation RK method is not smaller than that of the
original RK method if
\begin{equation}
  0 \le \gamma \le \frac{-1}{R(-\ssp) - 1} \geq 1,
\end{equation}
where $R$ is the stability function of the explicit SSP Runge--Kutta
method \eqref{eq:RK-step}. We also have the following
\begin{theorem}
  Consider an explicit SSP Runge--Kutta method \eqref{eq:RK-step}
  with SSP coefficient $\ssp$.
  The corresponding relaxation RK method \eqref{eq:relaxation} with $m=1$
  has SSP coefficient $\ssp_\gamma = \ssp + \O( \dt^{p-1} )$.
\end{theorem}
\begin{proof}
  The stability function $R$ of an $s$-stage Runge--Kutta method with
  Butcher coefficients $A, b$ is
  \begin{equation}
    R(z) = 1 + z b^T (\I - z A)^{-1} \1,
  \end{equation}
  where $\1 = (1, \dots, 1)^T \in \R^s$. The stability function of the
  relaxation method is given by $R_\gamma(z) = 1 + \gamma (R(z) - 1)$.
  The relaxation method is SSP with SSP coefficient $\geq \ssp_\gamma$
  if $R_\gamma(-\ssp_\gamma) \geq 0$,
  cf.\ \cite[Lemma~3.2 and Theorem~3.3]{ketcheson2019relaxation}.
  Because $\gamma = 1 + \O( \dt^{p-1} )$, this implies
  $R_\gamma(-\ssp) = R(-\ssp) + \O( \dt^{p-1} )$.
\end{proof}

Let us turn now to linear multistep methods.
To maintain the SSP property when the step size is not fixed,
we require that any coefficients of the fixed step size method
that are equal to zero remain exactly zero when the step size is varied:
\begin{subequations} \label{zero-assumption}
\begin{align}
    \alpha_i & = 0 \implies \alpha^n_i = 0, \\
    \beta_i & = 0 \implies \beta^n_i = 0.
\end{align}
\end{subequations}
This assumption is satisfied by the methods
of \cite{hadjimichael2016strong,mohammadi2019polynomial}:

\begin{theorem}
  Consider an explicit SSP LMM \eqref{eq:LMM} with SSP coefficient
  $\ssp$ and satisfying \eqref{zero-assumption}.
  If $m$ and $\nu_i$ are chosen such that $\nu_{i-m} > 0 \implies \alpha_{i-k} > 0$,
  then for small enough $\dt$
  the relaxation LMM \eqref{eq:relaxation} is SSP with an SSP coefficient
  $\ssp_\gamma = \ssp + \O( \dt^{p-1} )$ and
  $\ssp_\gamma \geq \ssp$ if $\gamma \in [0,1]$.
  Furthermore, if $m=k$ and $\nu_i = \alpha^n_i$ then $\ssp_\gamma = \ssp/\gamma$.
\end{theorem}
\begin{proof}
  Here, we use the definition of the SSP coefficient given as
  $\ssp_n$ in \cite{hadjimichael2016strong}, which can be written as
  \begin{equation}
    \ssp =
    \begin{cases}
      \max\set{ r \in [0, \infty) | \forall i\colon \alpha_i - r \beta_i \geq 0 },
      &\text{if } \forall i\colon \alpha_i, \beta_i \geq 0,
      \\
      0,
      &\text{otherwise}.
    \end{cases}
  \end{equation}
  The relaxation LMM can be interpreted as an LMM with parameters
  \begin{equation}
    \alpha_{i,\gamma}
    =
    \nu_i (1 - \gamma) + \gamma \alpha_i,
    \quad\text{and}\quad
    \beta_{i,\gamma} = \gamma \beta_i.
  \end{equation}
  Using $\gamma = 1 + \O( \dt^{p-1} )$,
  $\alpha_{i,\gamma} = \alpha_i + \O( \dt^{p-1} )$ and
  $\beta_{i,\gamma} = \beta_i + \O( \dt^{p-1} )$.
  For small enough $\dt$ these coefficients are also non-negative.
  Hence, $\ssp_\gamma = \ssp + \O( \dt^{p-1} )$.

  Meanwhile, if $m=k$ and $\nu_i = \alpha_i$ then the relaxation LMM can be written
  as a standard LMM but with coefficients $(\alpha,\gamma\beta)$.
  Hence the SSP coefficient is $\ssp/\gamma$.
\end{proof}

Choosing a variable step size method that does not satisfy the implication
$\nu_{i-m} > 0 \implies \alpha_{i-k} > 0$ can lead to loss of the SSP property.
For instance, the second-order three-step method with variable step size
of \cite{hadjimichael2016strong} is given by
\begin{equation}
  \unew
  =
  \frac{\Omega_2^2 - 1}{\Omega_2^2} \left(
    u^{n-1} + \frac{\Omega_2}{\Omega_2 - 1} \dt f(u^{n-1})
  \right)
  + \frac{1}{\Omega_2^2} u^{n-3},
\end{equation}
where $\Omega_2 = (t^{n-1} - t^{n-3}) / \dt$.
Taking e.g. $m=2$ and $\nu_0\ne 0$ generates a term $(1-\gamma)u^{n-2}$
in the relaxation solution, which destroys the SSP property when $\gamma>1$.

In general, orthogonal projection methods can violate SSP properties.
Indeed, linear functionals are convex and linear invariants are preserved
by the explicit Euler method (as well as by all Runge--Kutta, linear
multistep, and general linear methods).
Since orthogonal projection methods do not conserve linear invariants
in general, the corresponding convex stability property is lost.
\begin{example}
  Consider the two-step second-order SSP Runge--Kutta method
  SSPRK(2,2) given by
  \begin{equation}
  \begin{aligned}
    y^{1} &= u^{n-1}, \\
    y^{2} &= u^{n-1} + \dt f(y^{1}), \\
    \unew &= u^{n-1} + \frac{1}{2} \dt \left( f(y^{1}) + f(y^{2}) \right).
  \end{aligned}
  \end{equation}
  Solutions $u$ of the ODE
  \begin{equation}
    \od{}{t} u(t)
    =
    \begin{pmatrix}
      0 & -1 & 1 \\
      1 & 0 & -1 \\
      -1 & 1 & 0
    \end{pmatrix}
    u(t),
    \quad
    u(0) = \begin{pmatrix} -1 \\ 0 \\ 0 \end{pmatrix},
  \end{equation}
  have a constant energy $\eta(u) = \frac{1}{2} \| u \|^2$ and total
  mass $\mathcal{M}(u) = \sum_i u_i$. The first step of the orthogonal
  projection SSPRK(2,2) method results in the total mass
  \begin{equation}
    \mathcal{M}( u^{1}_\lambda )
    =
    - \frac{ \sqrt{2} }{ \sqrt{2 + 3 \dt^4} }
    >
    \mathcal{M}(u^{0}).
  \end{equation}
  Hence, the convex stability property related to the total mass is violated.
  In contrast, the relaxation SSPRK(2,2) method preserves the total mass.
\end{example}

\section{Numerical Results}
\label{sec:numerical-results}

Here, the following classes of linear multistep methods \eqref{eq:LMM}
are considered. If not stated otherwise, the estimate $\etanew$ is
obtained using a dense output formula and Gauß quadrature using one ($k=2$)
or two ($k \in \{3,4\}$) nodes.
\begin{itemize}
  \item Adams($k$)\\
  The $k$-step explicit Adams methods (also known as Adams--Bashforth methods)
  are based on the formula
  $u^{n} = u^{n-1} + \int_{t^{n-1}}^{t^{n}} \mathcal{P}_f$,
  where $\mathcal{P}_f$ is the polynomial interpolating
  $f(u^{n-1})$, \dots, $f(u^{n-k})$,
  see \cite{bashforth1883attempt} and \cite[Section~III.1]{hairer2008solving}.
  These methods can be used with variable step sizes.
  A natural dense output at an intermediate value $\tau_i$ is generated by
  evaluating the integral with upper limit $\tau_i$ instead of $t^{n}$.

  \item Nyström($k$)AS\\
  The $k$-step Nyström methods are based on the formula
  $u^{n} = u^{n-2} + \int_{t^{n-2}}^{t^{n}} \mathcal{P}_f$,
  where $\mathcal{P}_f$ is again the polynomial interpolating $f(u^{n-1})$, \dots,
  $f(u^{n-k})$,
  see \cite{nystrom1925numerische} and \cite[Section~III.1]{hairer2008solving}.
  Based on these constant step size Nyström methods, an extension to variable
  step sizes that is equipped with a dense output formula has been proposed
  by Arévalo and Söderlind \cite{arevalo2017grid} and will be denoted as
  Nyström($k$)AS.

  \item eBDF($k$), eBDF($k$)AS\\
  The family of extrapolated backward difference formula (eBDF) methods
  is based on the formula
  $\mathcal{P}_u'(t^{n}) = \mathcal{P}_f(t^{n})$,
  where $\mathcal{P}_u$ and $\mathcal{P}_f$ are polynomials that
  interpolate the previous step
  values $u^{n-i}$ and step derivatives $f(u^{n-i})$, respectively,
  cf.\ \cite{ruuth2005high}.
  Based on the constant step size eBDF methods, an extension to variable
  step sizes that is equipped with a dense output formula has been proposed
  by Arévalo and Söderlind \cite{arevalo2017grid} and will be denoted as
  eBDF($k$)AS.

  \item SSP($k, p$), SSP($k, p$)AS\\
  Second and third order accurate variable step size SSP LMMs have been
  proposed in \cite{hadjimichael2016strong}. The estimate of the evolution
  of $\eta$ can either be based on the evolution predicted by the SSP
  method itself \eqref{eq:eta-est-1-SSPLMM} or on the quadrature
  \eqref{eq:eta-est-1-quadrature} using the dense output formula of
  \cite{mohammadi2019polynomial}, which is based on the framework of
  \cite{arevalo2017grid}. If the latter option is chosen, the method is
  denoted as SSP($k, p$)AS.

  \item EDC($i$,$j$)\\
  The explicit difference correction (EDC) methods of \cite{arevalo2000regular}
  are extended to variable step sizes and equipped with a dense output
  using the approach of \cite{arevalo2017grid}.

  \item BDF($k$)\\
  The family of backward difference formula (BDF) methods
  is based on the formula $p_u'(t^{n}) = f(\mathcal{P}_u(t^{n}))$,
  where $\mathcal{P}_u$ is a polynomial that interpolates the step
  values $u^{n}, u^{n-1}$, \dots, $u^{n-k}$, cf.\ \cite{curtiss1952integration}
  and \cite[Section~III.1]{hairer2008solving}.
\end{itemize}
If not stated otherwise, relaxation LMMs have been adapted to the new
step sizes using $m = 1$ and $\nu_i = \delta_{i,0}$ in \eqref{eq:old-values}.
We have checked that the results using different choices of $\nu_i$
are similar.
Since Theorem~\ref{thm:fixed-fixed-coefficients} can be applied to
Adams methods, Nyström methods, and SSP LMMs, we have also tested the
corresponding fixed coefficient version of these schemes.

\begin{remark}
  The modified leapfrog method of \cite{sanz1982explicit} is the relaxation
  Nyström($2$)AS method for conservative inner-product norms with $m = 2$
  and $\nu_i = \delta_{i,0}$ in \eqref{eq:old-values} and fixed step sizes,
  cf.\ Theorem~\ref{thm:fixed-fixed-coefficients}.
\end{remark}

We have implemented the relaxation methods used in this article
in Python, using SciPy \cite{virtanen2019scipy} to solve the scalar
non-quadratic equations for the relaxation parameter $\gamma$.
Matplotlib \cite{hunter2007matplotlib} has been used to generate the plots.
The source code for all numerical examples is available online
\cite{ranocha2020generalRepro}.

\subsection{Nonlinear Oscillator}
\label{sec:nonlinear-osc}

For the nonlinear oscillator
\begin{equation}
\label{eq:nonlinear-osc}
  \od{}{t} \begin{pmatrix} u_1(t) \\ u_2(t) \end{pmatrix}
  =
  \| u(t) \|^{-2} \begin{pmatrix} -u_2(t) \\ u_1(t) \end{pmatrix},
  \quad
  u^{0}
  =
  \begin{pmatrix} 1 \\ 0 \end{pmatrix},
\end{equation}
of \cite{ranocha2020strong,ranocha2020energy}, the energy
$\eta(u) = \frac{1}{2} \| u \|^2$ is conserved.
We choose this test problem to demonstrate the convergence properties of
relaxation methods in a simple setting where the relaxation parameter can
also be computed explicitly by solving a quadratic equation. We use some common
explicit methods here because this problem is not stiff.

Results of a convergence study for this problem are visualized in
Figure~\ref{fig:nonlinear-osc-convergence}.
The Nyström($k$)AS, $k \in \{ 3, 4\}$, methods result in a large error and
are not completely in the asymptotic regime, which could be attributed to
their lack of a reasonable stability region \cite[Section~V.1]{hairer2010solving}.
However, applying projection or relaxation results in the expected order
of accuracy.
The Adams methods do not have similar problems and work well. All other
explicit methods described above behave similarly to the Adams methods.
For this test problem and formally odd-order relaxation and projection
methods, there is a certain superconvergence phenomenon, increasing the
experimental order of accuracy by one in accordance with
\arXiv{%
the analysis given in the extended version of this article available
on arXiv.org \cite{ranocha2020general}.}{%
Theorem~\ref{thm:nonlinear-Euclidean-Hamiltonian}.
}

\begin{figure}[ht]
\centering
  \begin{subfigure}{\textwidth}
  \centering
    \includegraphics[width=\textwidth]{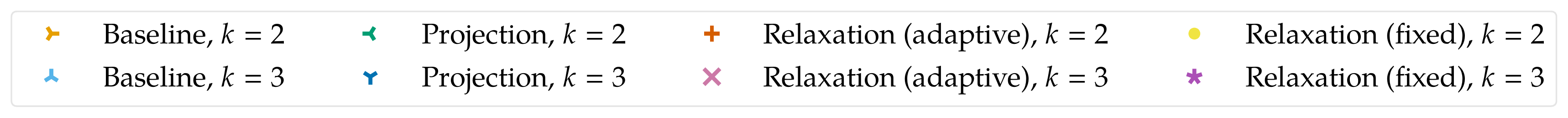}
  \end{subfigure}%
  \\
  \begin{subfigure}{0.49\textwidth}
  \centering
    \includegraphics[width=\textwidth]{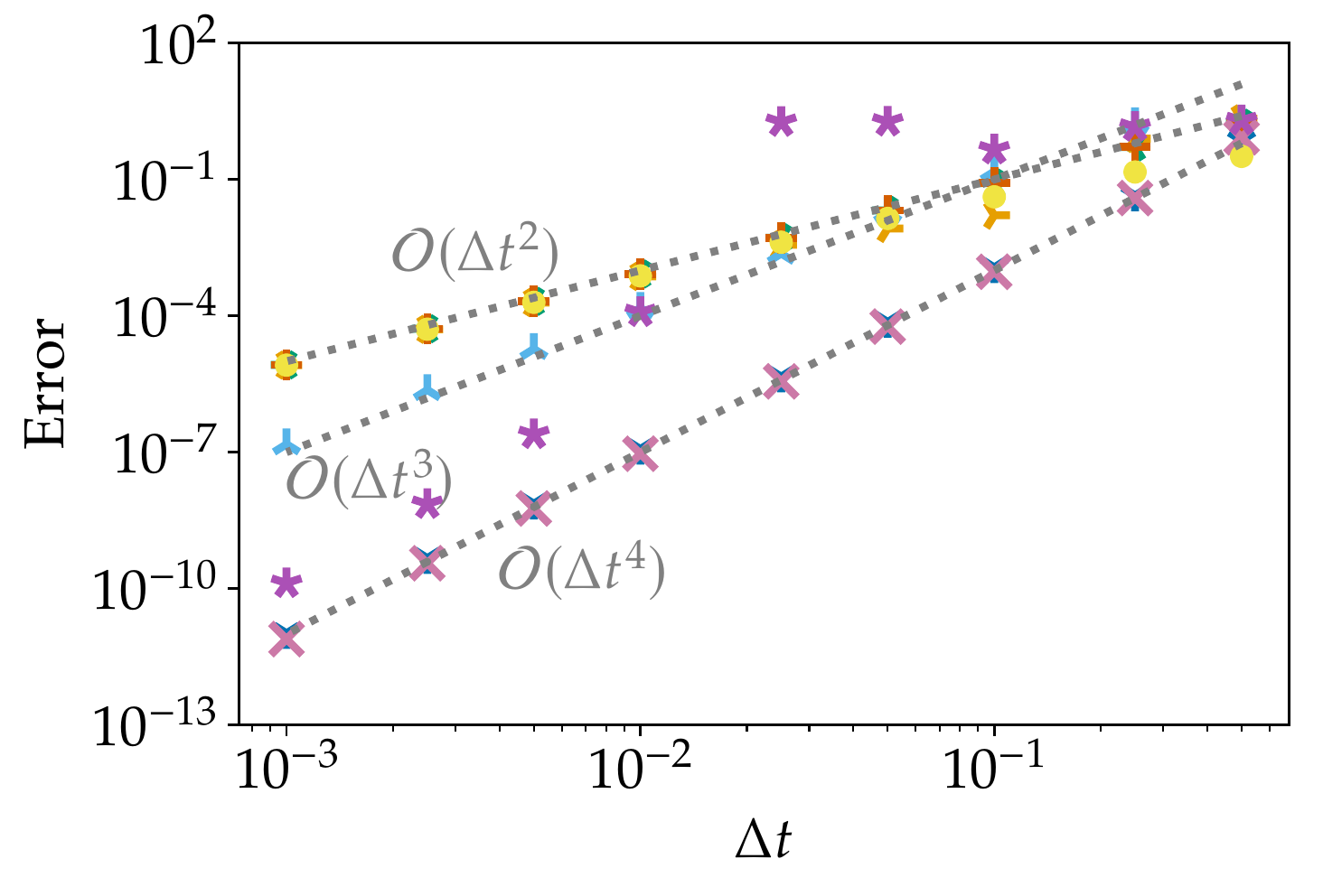}
    \caption{Adams($k$) methods.}
  \end{subfigure}%
  \hspace*{\fill}
  \begin{subfigure}{0.49\textwidth}
  \centering
    \includegraphics[width=\textwidth]{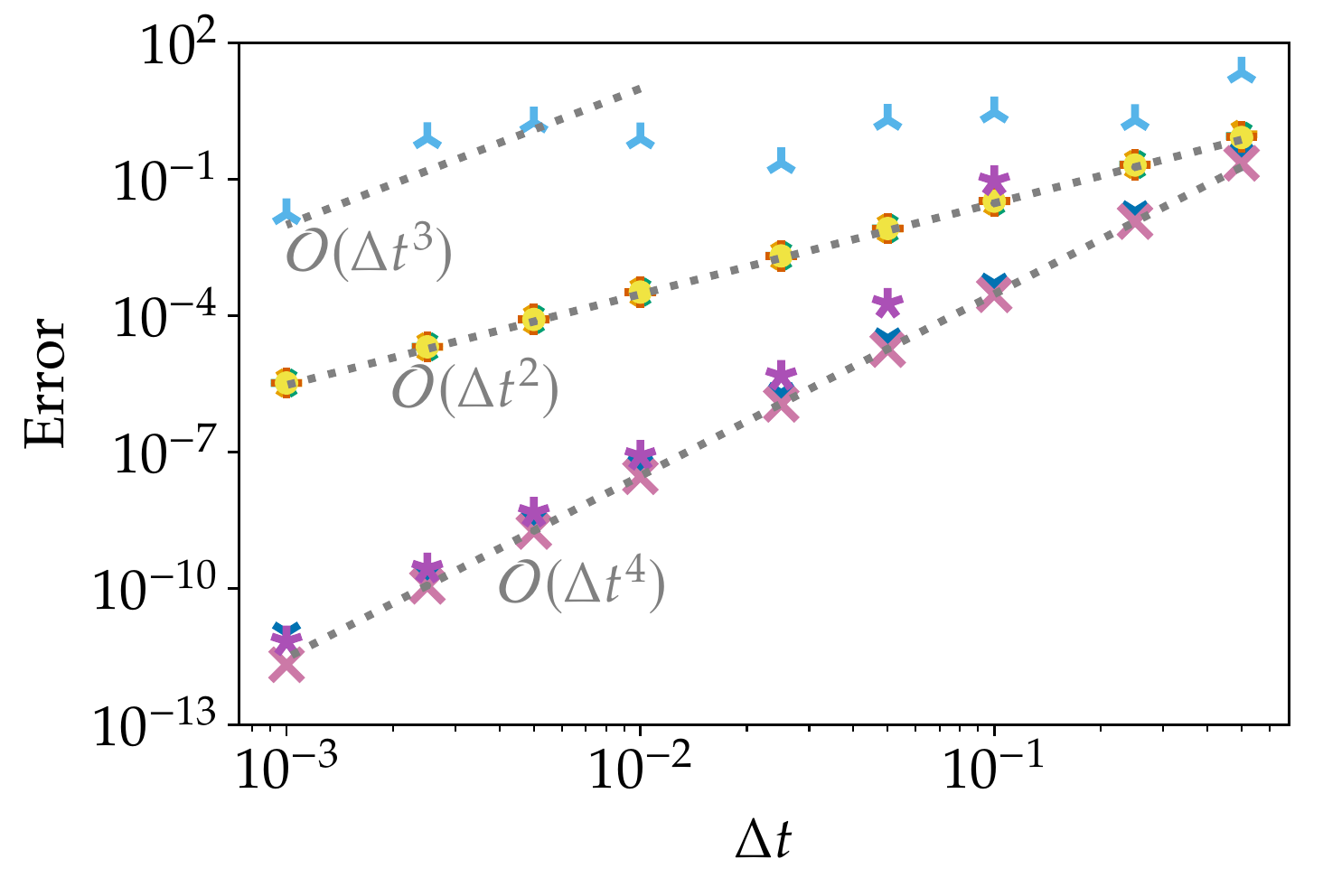}
    \caption{Nyström($k$)AS methods.}
  \end{subfigure}%
  \caption{Convergence study for linear multistep methods applied to the
           nonlinear oscillator \eqref{eq:nonlinear-osc} with final time
           $t = 20$.}
  \label{fig:nonlinear-osc-convergence}
\end{figure}

The fixed coefficient versions of Adams($3$) and Nyström($3$)AS result in
larger errors than the corresponding versions with adapted coefficients.
For smaller time steps $\dt$, they are even not in the asymptotic regime.
This behavior is in accordance with the analysis of
Section~\ref{sec:fixed-fixed-coefficients}.
The energy evolution of a representative example from this section is shown
in Figure~\ref{fig:nonlinear-osc-energy}.

\subsection{Kepler Problem}
\label{sec:kepler}

The Kepler problem
\begin{equation}
\label{eq:kepler}
\begin{gathered}
  \od{}{t} q(t)
  = \od{}{t} \begin{pmatrix} q_1(t) \\ q_2(t) \end{pmatrix}
  = p(t),
  \quad
  \od{}{t} p_i(t) = - \frac{q_i(t)}{ \abs{q(t)}^3 },
  \\
  q(0) = \begin{pmatrix} 1-e \\ 0 \end{pmatrix},
  \quad
  p(0) = \begin{pmatrix} 0 \\ \sqrt{ (1-e) / (1+e) } \end{pmatrix},
\end{gathered}
\end{equation}
with eccentricity $e = 0.5$ is a Hamiltonian system
\begin{equation}
\label{eq:Hamiltonian}
  \od{}{t} q(t) = \partial_p H\bigl( q(t), p(t) \bigr),
  \quad
  \od{}{t} p(t) = - \partial_q H\bigl( q(t), p(t) \bigr),
\end{equation}
with Hamiltonian
\begin{equation}
  H(q,p) = \frac{1}{2} \abs{p}^2 - \frac{1}{\abs{q}},
\end{equation}
where the angular momentum
\begin{equation}
  L(q,p) = q_1 p_2 - q_2 p_1
\end{equation}
is an additional conserved functional, cf.\
\cite[Section~1.2.4]{sanzserna1994numerical}.
We choose this test problem to demonstrate the convergence properties of
relaxation methods in a more complex setting where the relaxation parameter
is computed using a scalar root finding method. Since this problem is
not stiff, we use explicit methods. We choose LMMs different from the ones
used in Section~\ref{sec:nonlinear-osc} because we want to demonstrate the
applicability of relaxation methods for a variety of schemes.

The baseline, projection, and relaxation variants of the explicit multistep
methods described above converge with the expected order of accuracy for
this problem if the energy or angular momentum is conserved by the
projection/relaxation method.
As examples, third- and fourth-order accurate eBDF methods yield the
convergence results shown in Figure~\ref{fig:kepler-convergence}. Clearly,
both the projection and the relaxation methods reduce the error compared to
the baseline schemes. The results for the other explicit methods described
above are similar.
The energy evolution of a representative example from this section is shown
in Figure~\ref{fig:kepler-energy}.

\begin{figure}[ht]
\centering
  \includegraphics[width=0.7\textwidth]{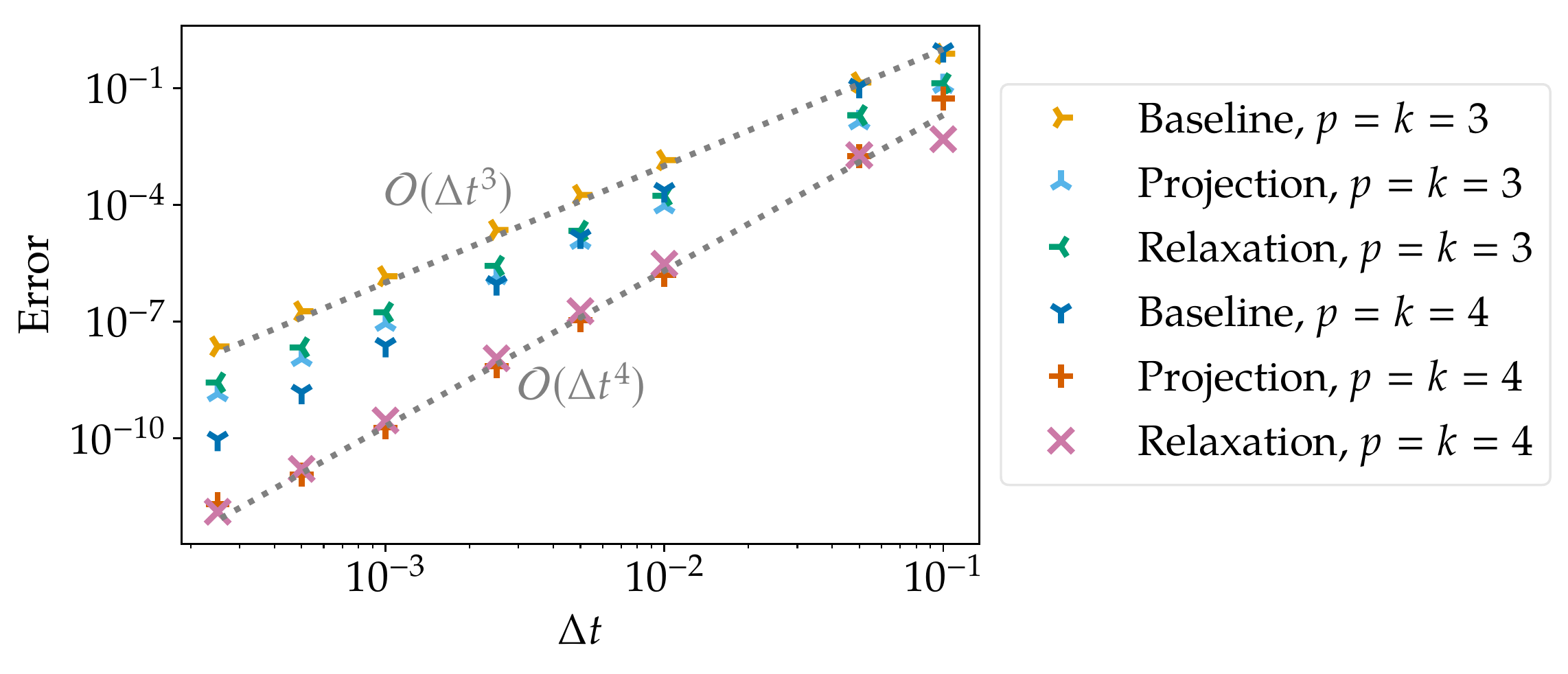}
  \caption{Convergence study for eBDF methods applied to the
           Kepler problem \eqref{eq:kepler} with final time $t = 5$ and
           projection/relaxation methods conserving the energy. The results
           for methods conserving the angular momentum are very similar.}
  \label{fig:kepler-convergence}
\end{figure}

\begin{figure}[ht]
\centering
  \begin{subfigure}{0.49\textwidth}
  \centering
    \includegraphics[width=\textwidth]{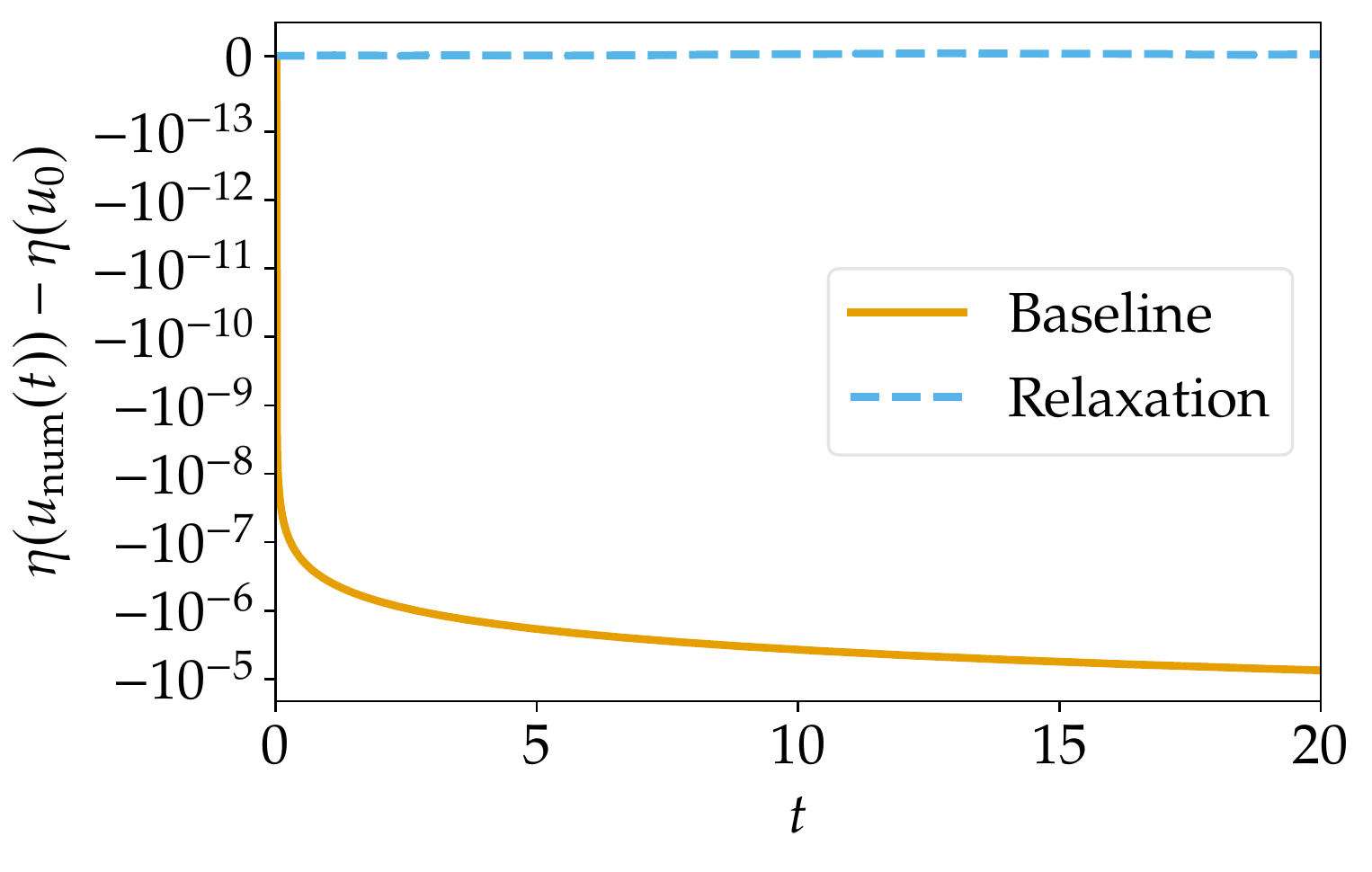}
    \caption{Nonlinear oscillator \eqref{eq:nonlinear-osc},
             $\dt = 0.01$.}
    \label{fig:nonlinear-osc-energy}
  \end{subfigure}%
  \hspace*{\fill}
  \begin{subfigure}{0.49\textwidth}
  \centering
    \includegraphics[width=\textwidth]{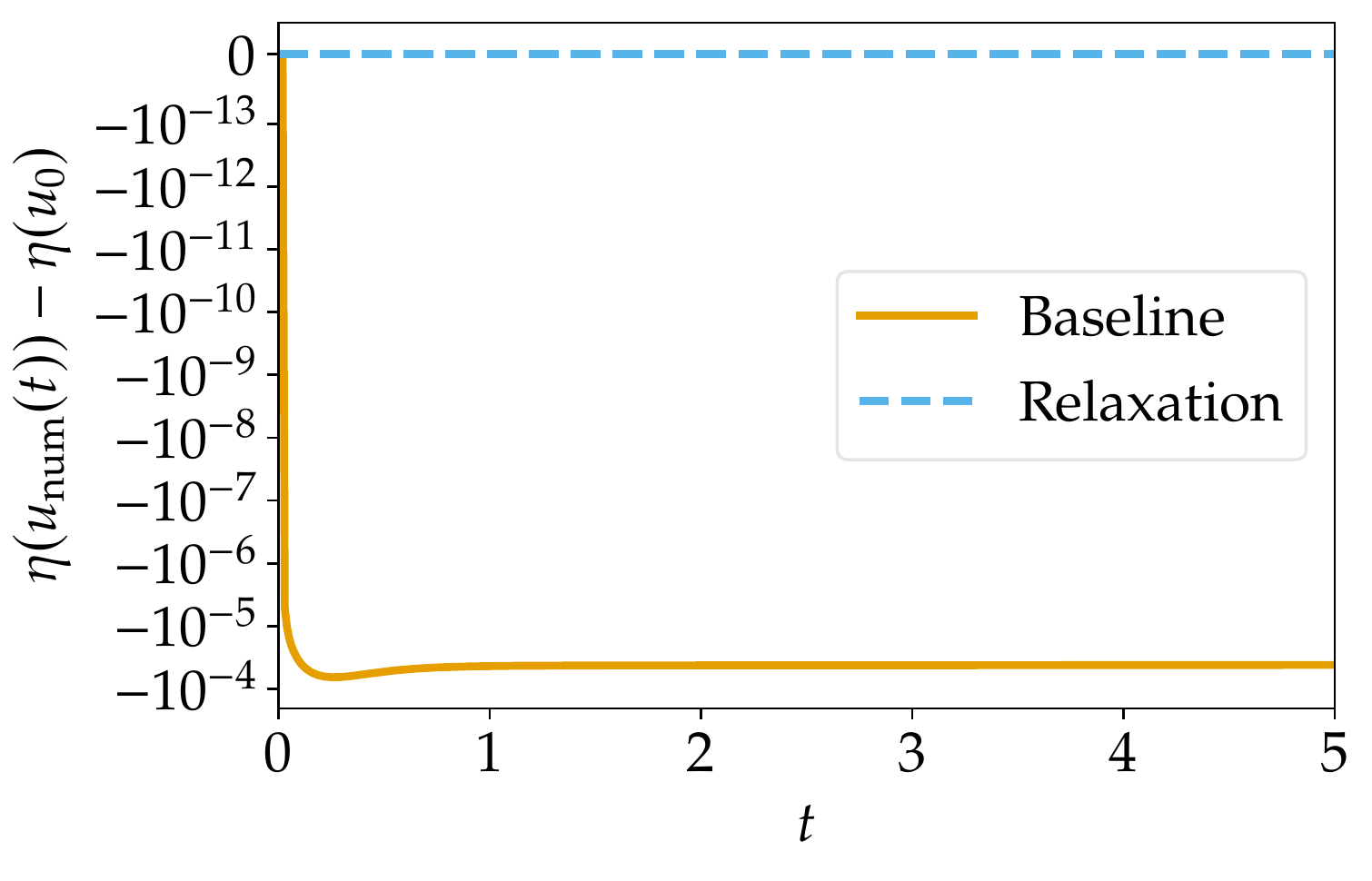}
    \caption{Kepler problem \eqref{eq:kepler},
             $\dt = 0.01$.}
    \label{fig:kepler-energy}
  \end{subfigure}%
  \caption{Energy evolution of Adams(3) methods with and without relaxation
           (adapted time step and coefficients) for representative examples from
           Section~\ref{sec:nonlinear-osc} and \ref{sec:kepler}.}
\end{figure}

\subsection{Dissipated Exponential Entropy}
\label{sec:diss_exp}

Consider the ODE
\begin{equation}
\label{eq:diss_exp}
  \od{}{t} u(t)
  =
  -\exp(u(t)),
  \quad
  u^{0}
  =
  0.5,
\end{equation}
with exponential entropy $\eta(u) = \exp(u)$,
which is dissipated for the analytical solution
\begin{equation}
  u(t)
  =
  -\log\bigl( \e^{-1/2} + t \bigr).
\end{equation}
In contrast to the conservative problems described above, this problem is
dissipative. We choose this test problem to demonstrate the convergence
properties of relaxation methods also in this setting. Since the test problem
is still not stiff, we choose a variety of explicit methods.

Again, the explicit multistep methods described above with or without
projection/relaxation converge with the expected order of accuracy for
this problem.
As examples, third- and fourth-order accurate multistep methods yield the
convergence results shown in Figure~\ref{fig:diss_exp-convergence}.
There is no significant difference between the two different estimates
\eqref{eq:eta-est-1-SSPLMM} and \eqref{eq:eta-est-1-quadrature} for SSP($4,3$).

\begin{figure}[htp]
\centering
  \begin{subfigure}{0.7\textwidth}
  \centering
    \includegraphics[width=\textwidth]{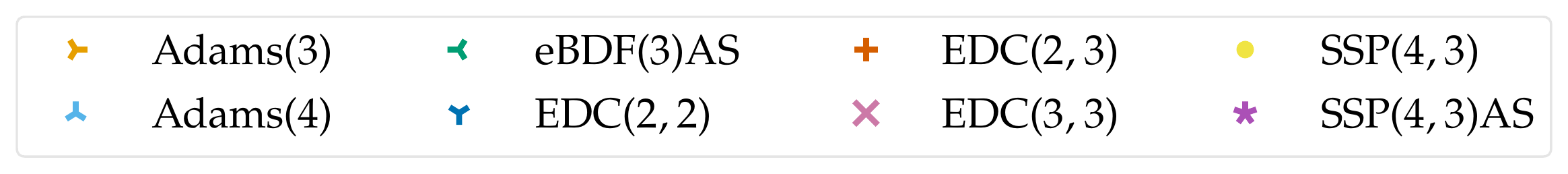}
  \end{subfigure}%
  \\
  \begin{subfigure}{0.33\textwidth}
  \centering
    \includegraphics[width=\textwidth]{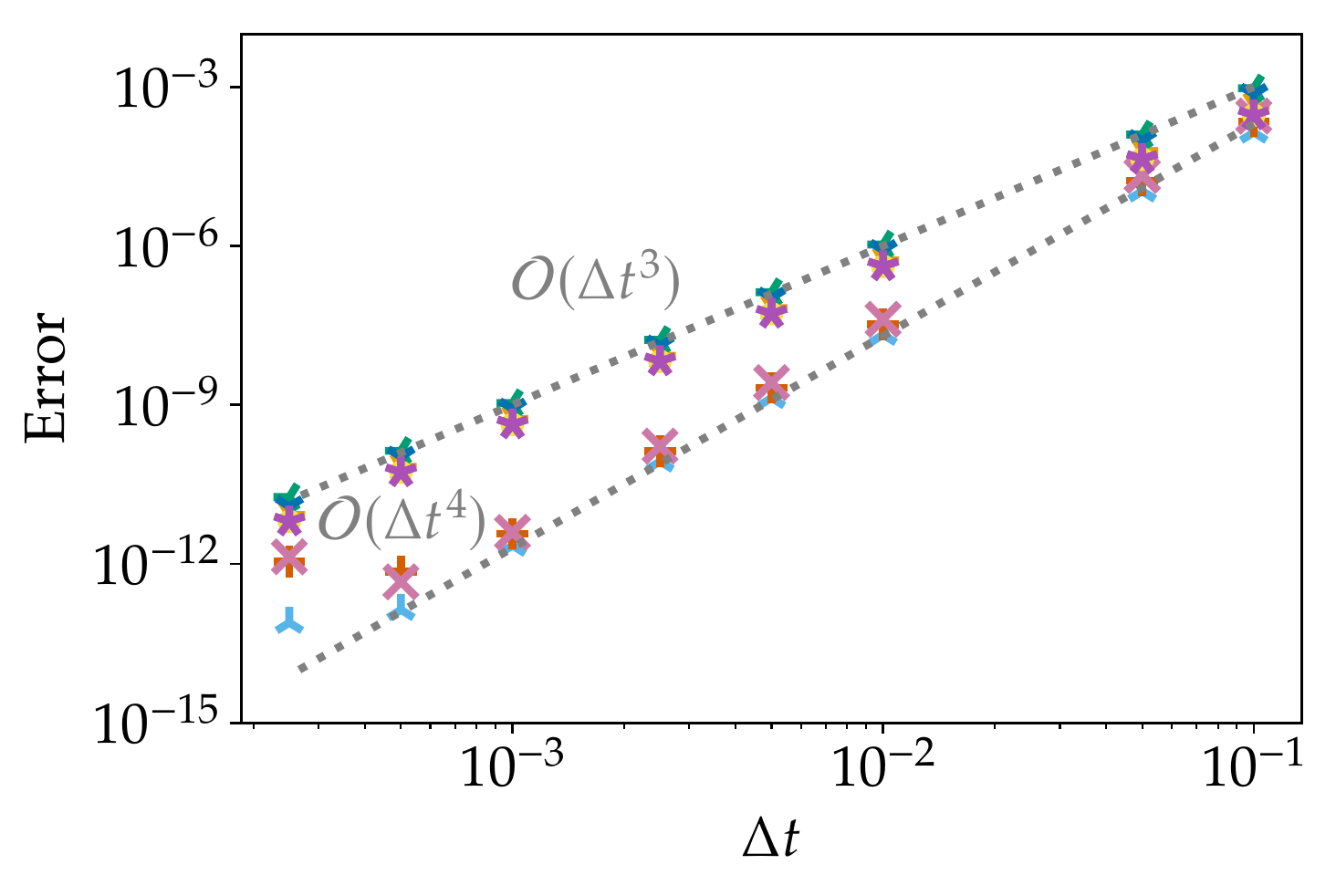}
    \caption{Baseline methods.}
  \end{subfigure}%
  \begin{subfigure}{0.33\textwidth}
  \centering
    \includegraphics[width=\textwidth]{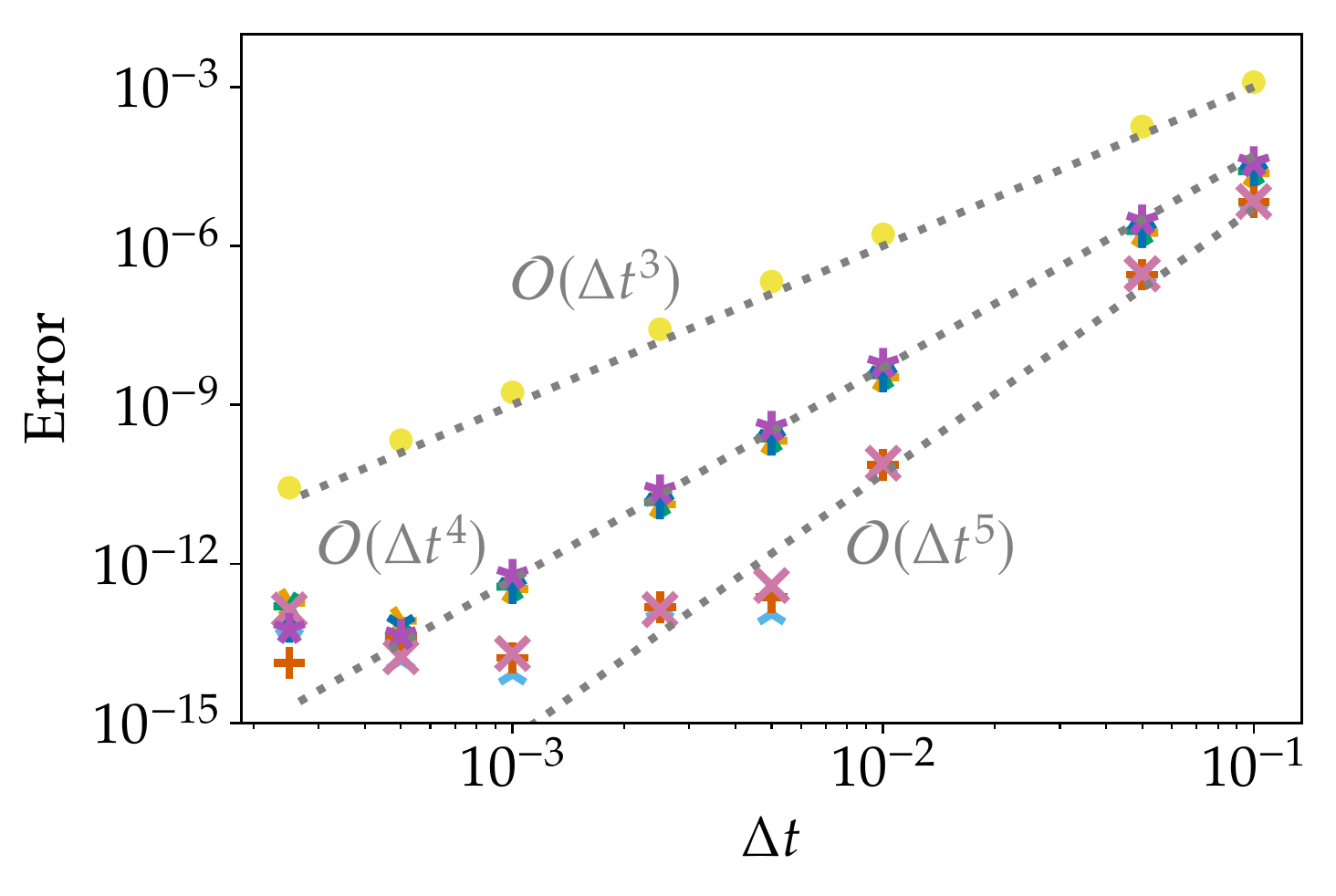}
    \caption{Projection methods.}
  \end{subfigure}%
  \begin{subfigure}{0.33\textwidth}
  \centering
    \includegraphics[width=\textwidth]{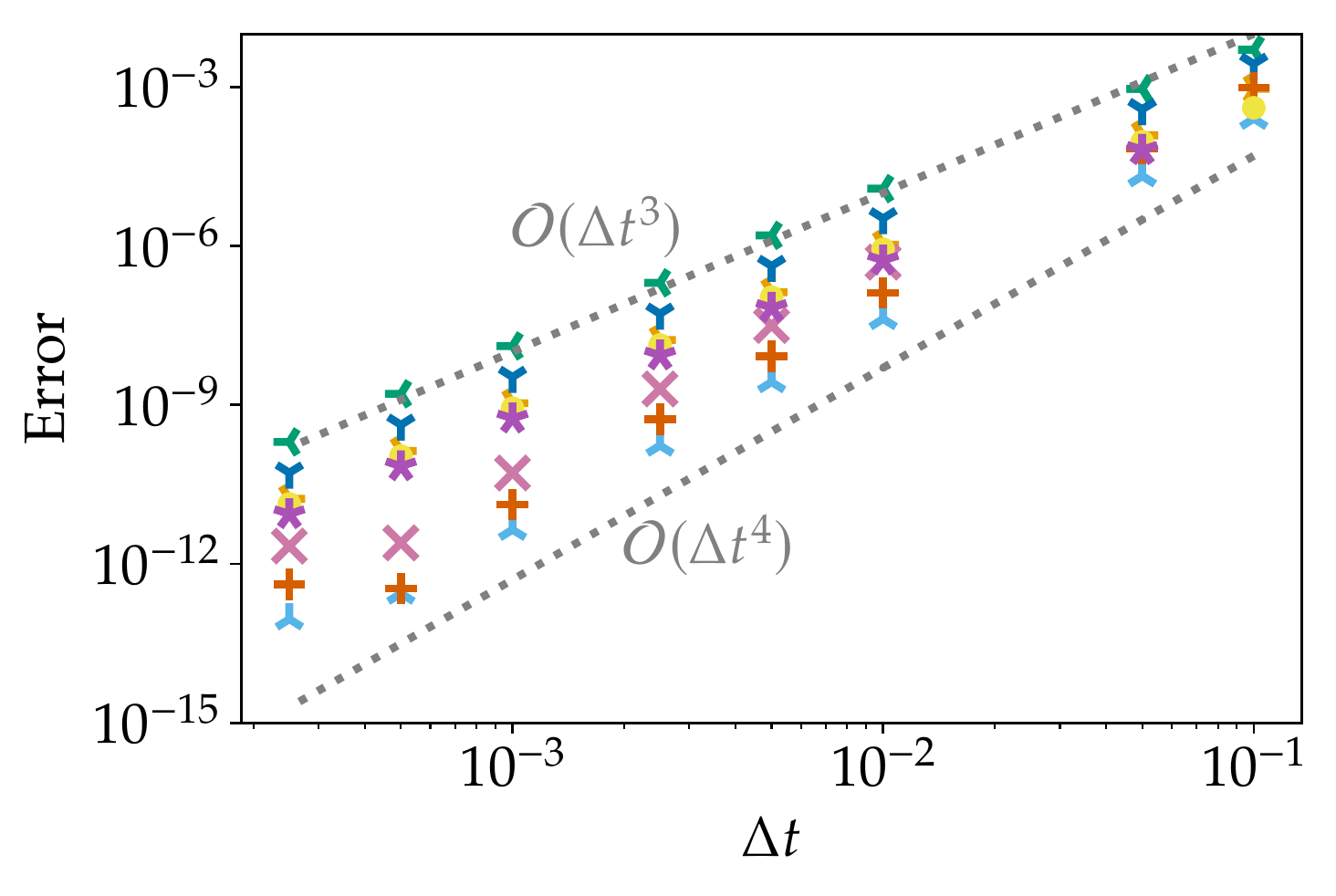}
    \caption{Relaxation methods.}
  \end{subfigure}%
  \caption{Convergence study for linear multistep methods applied to the
           ODE \eqref{eq:diss_exp} with dissipated exponential entropy,
           final time $t = 20$, and projection/relaxation methods based on
           the exponential entropy.}
  \label{fig:diss_exp-convergence}
\end{figure}

Results of fixed step size SSP LMMs with $\nu_i = \alpha_i$
applied to the ODE \eqref{eq:diss_exp} with dissipated exponential entropy
are shown in Figure~\ref{fig:t_tau_diss_exp}.
Because of the exact starting procedure, $t = \tau$ for the first $k$
steps of a $k$-step method. Thereafter, $|t - \tau|$ increases in time.
As discussed in Section~\ref{sec:fixed-fixed-coefficients},
$\max_\tau |t - \tau| = \O( \dtau^{p-2} )$. This can also be
observed for SSP($3, 2$), where the final value of $t - \tau$ is independent
of the time step, and SSP($4, 3$), where the final value of $t - \tau$
decreases proportionally to $\dtau$.
Since $\max_\tau |t - \tau| = \O(1)$ for SSP($3, 2$), the maximal effective
relaxation parameter $\Gamma(\tau) \gamma(\tau)$ is also $\O(1)$.

\begin{figure}[ht]
\centering
  \includegraphics[width=0.8\textwidth]{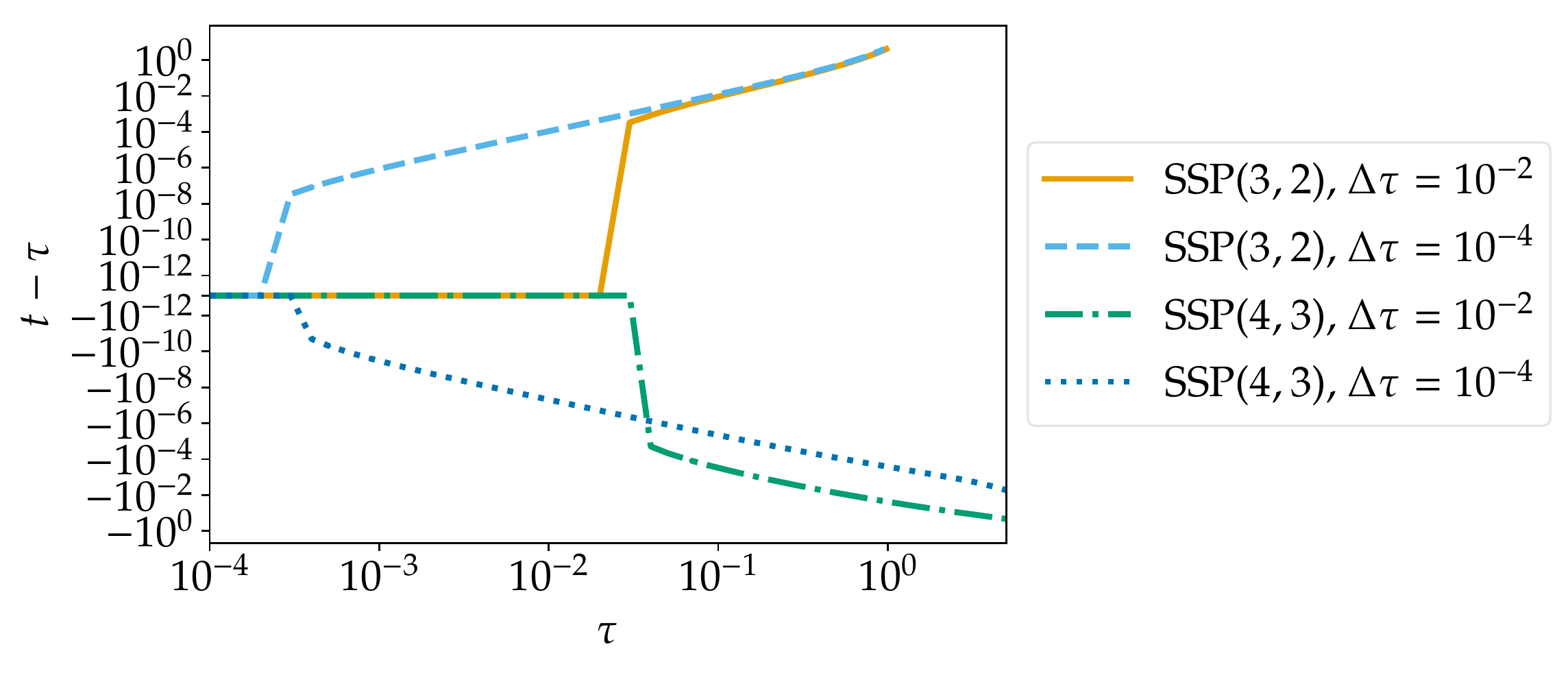}
  \caption{Difference of the physical time $t$ and the pseudotime $\tau$
           for relaxation SSP LMMs with fixed step size $\dtau$ applied
           to the ODE \eqref{eq:diss_exp} with dissipated exponential
           entropy and final time $t = 5$.}
  \label{fig:t_tau_diss_exp}
\end{figure}

If $\nu_i \neq \alpha_i$, e.g.\ if $m = 1$ and $\nu_i = \delta_{i,0}$ as usual for
methods with adapted step sizes, no solution $\gamma > 0$ can be found
for this problem and the SSP LMMs with fixed step sizes.
The Adams($2$) method with fixed step sizes applied to this problem works well
if the final time is reduced to $t = 2.5$. For larger final times,
the error of the numerical solutions grows because of the growth of
$\Gamma(\tau)$. Then, the time step $\dtau$ has to be reduced to get
acceptable solutions. The Adams methods with adapted step sizes can
be applied successfully to this problem with much larger values of the
time step $\dt$, in accordance with the analysis of
Section~\ref{sec:fixed-fixed-coefficients}.

\subsection{Korteweg--de Vries Equation}
\label{sec:KdV}

The Korteweg--de Vries (KdV) equation
\begin{equation}
\label{eq:KdV}
  \partial_t u(t,x) + \partial_x \frac{u(t,x)^2}{2} + \partial_x^3 u(t,x) = 0
\end{equation}
is well-known in the literature as a nonlinear PDE which admits soliton
solutions of the form
\begin{equation}
  u(t,x) = A \cosh\bigl( \sqrt{3 A} (x - c t - \mu) / 6 \bigr)^{-2},
  \quad
  c = \nicefrac{A}{3},
\end{equation}
where $A \geq 0$ is the amplitude, $c$ the wave speed, and $\mu$ an
arbitrary constant. The KdV equation possesses an infinite hierarchy
of conserved integral functionals, including the mass (with density $u$)
and the energy (with density $u^2$).

Numerical methods that conserve both the mass and the energy result in
an asymptotic error growth that is only linear in time, while other methods
will in general yield an asymptotically quadratic error growth
\cite{frutos1997accuracy}. If only the energy is conserved, the error is
usually reduced at first and the quadratic error growth can be seen later
than for methods that do not conserve the energy.

We choose this test problem because it is a stiff nonlinear problem. Hence,
we use implicit methods to deal with the stiffness. In addition, this problem
demonstrates improved qualitative properties of conservative schemes and
the importance of preserving linear invariants.
Moreover, this stiff problem demonstrates that important stability properties
are not lost by introducing relaxation, even for robust, $A$-stable methods.

Here, we use the mass- and energy-conservative Fourier collocation
semidiscretization described in \cite{ranocha2020relaxationHamiltonian}
with $N = 64$ modes in an interval of length $L = 80$ for the amplitude
$A = 2$.
Integrating the resulting stiff ODE in time with the BDF($2$) method and
a time step $\dt = 0.1$ yields the results shown in Figure~\ref{fig:KdV}.
The error for both the projection and the relaxation grows linearly in time
at first. For the projection method not conserving the total mass, the error
starts to grow quadratically shortly before it saturates (since there is no
overlap of the numerical solution and the analytical solution anymore).

\begin{figure}[htp]
\centering
  \begin{subfigure}{0.7\textwidth}
  \centering
    \includegraphics[width=\textwidth]{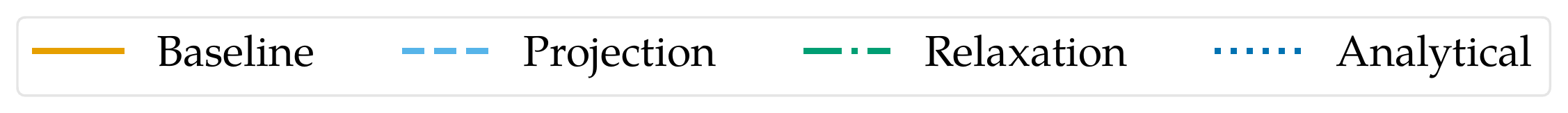}
  \end{subfigure}%
  \\
  \begin{subfigure}{0.49\textwidth}
  \centering
    \includegraphics[width=\textwidth]{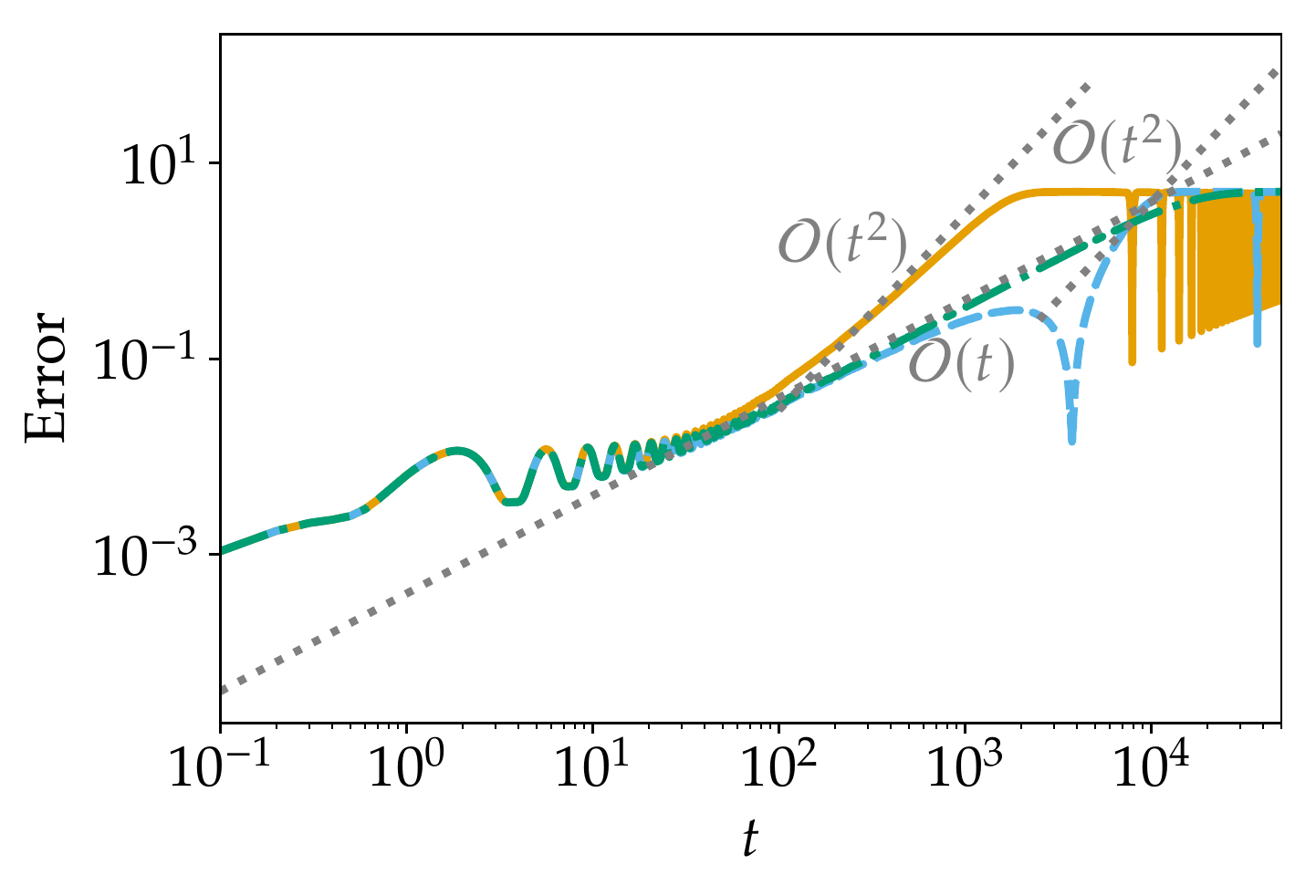}
    \caption{Error growth in time.}
  \end{subfigure}%
  ~
  \begin{subfigure}{0.49\textwidth}
  \centering
    \includegraphics[width=\textwidth]{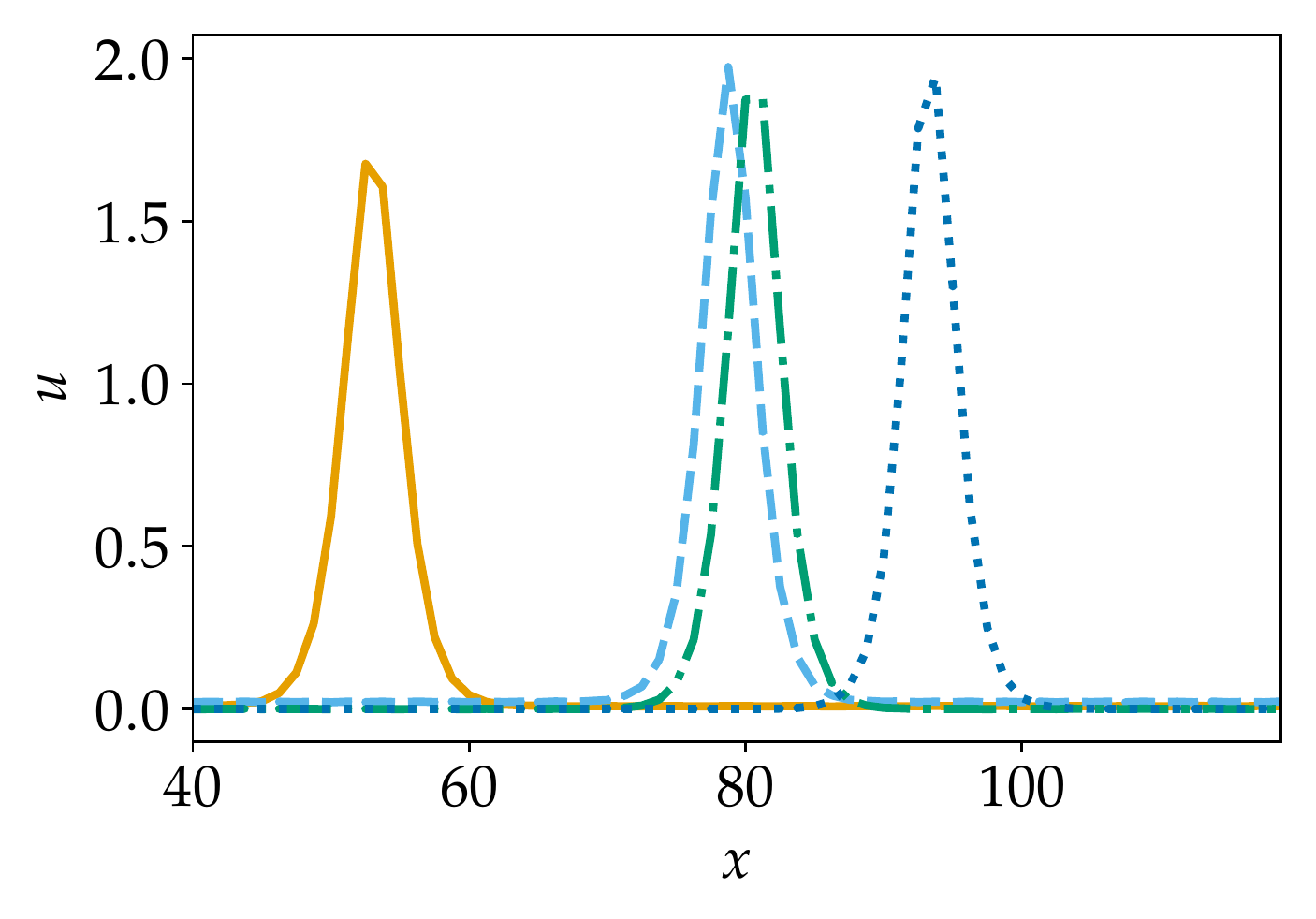}
    \caption{Numerical solutions at the final time.}
  \end{subfigure}%
  \\
  \begin{subfigure}{0.49\textwidth}
  \centering
    \includegraphics[width=\textwidth]{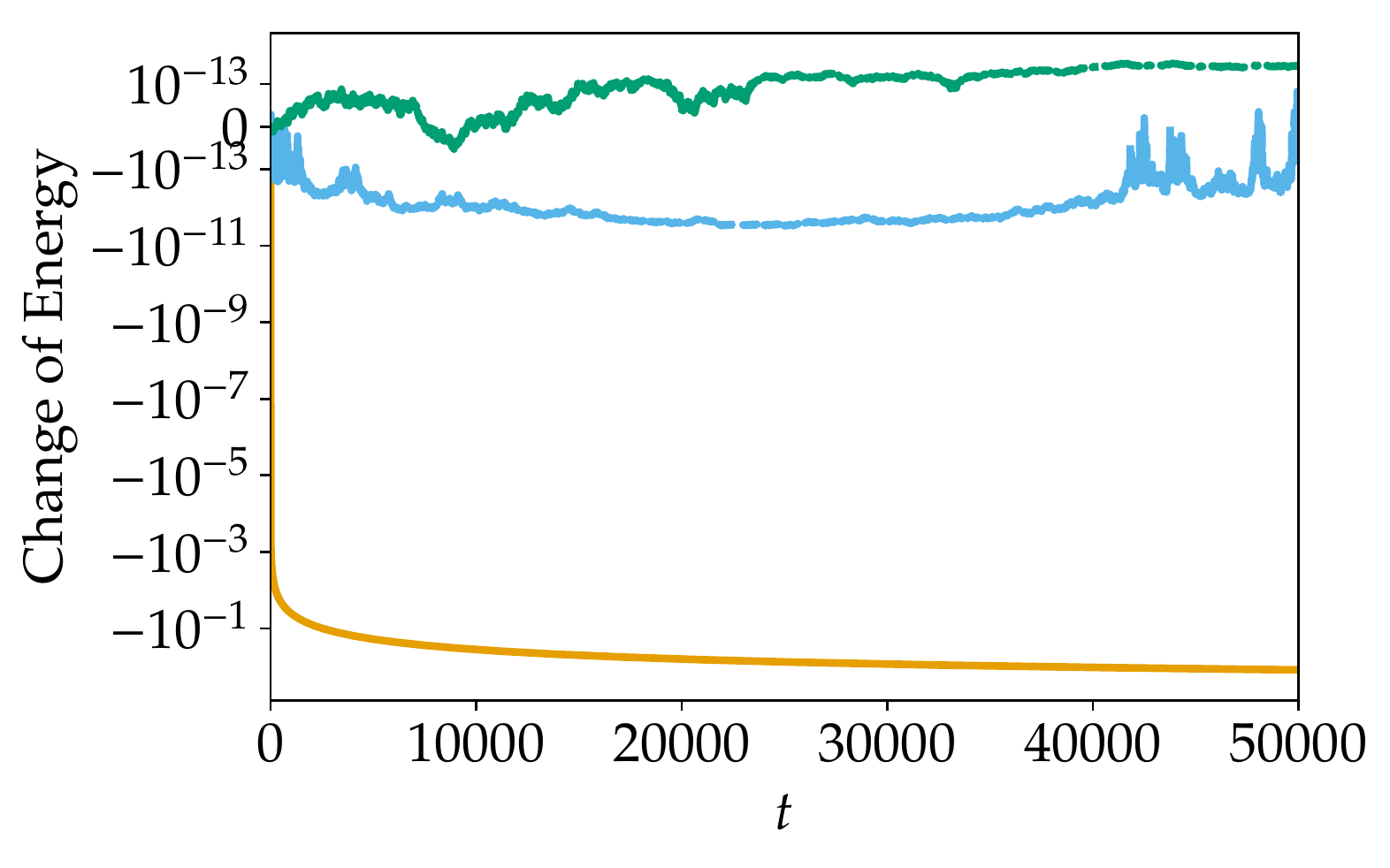}
    \caption{Change of the total energy.}
  \end{subfigure}%
  ~
  \begin{subfigure}{0.49\textwidth}
  \centering
    \includegraphics[width=\textwidth]{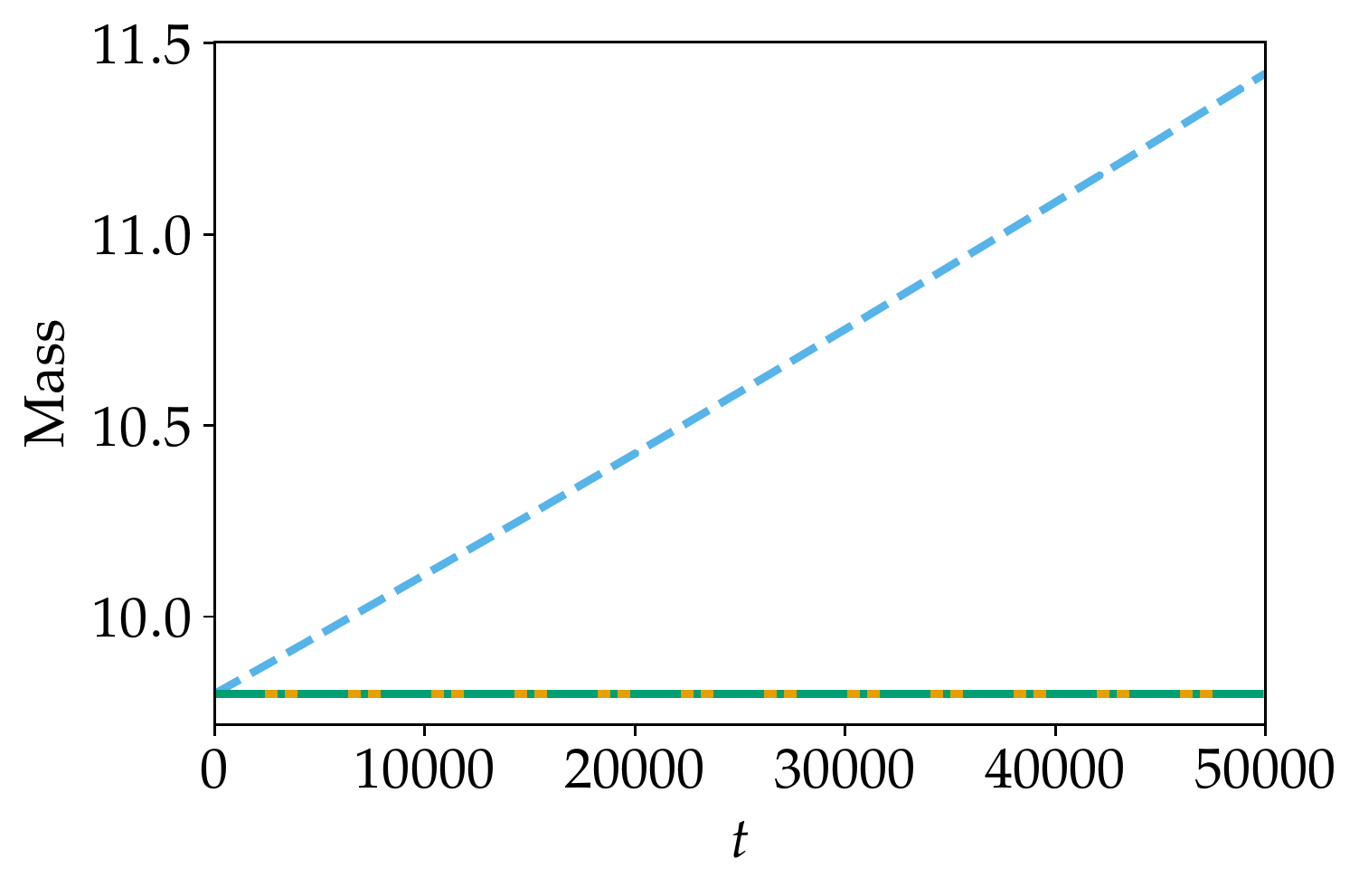}
    \caption{Change of the total mass.}
  \end{subfigure}%
  \caption{Numerical solutions of the KdV equation \eqref{eq:KdV} with final
           time $t = \num{5.0e4}$ and projection/relaxation methods conserving
           the energy applied to a mass- and energy-conservative Fourier
           collocation method. The baseline method is BDF2.}
  \label{fig:KdV}
\end{figure}

We would like to point out that we considered a long-time integration for this
stiff nonlinear PDE example. The (phase) error grows in time and will reach
\SI{100}{\percent} for every time step at some time (which increases for
decreasing step size $\dt$). In particular, having an error of \SI{100}{\percent}
after more than 400 periods of the traveling wave solution does not indicate
instabilities caused choosing the time step $\dt$ too big. Instead, this behavior
is expected and occurs also for smaller time steps (possibly after more periods).

\subsection{Compressible Euler Equations}
\label{sec:euler}

Here, we apply a second-order entropy-conservative finite difference method
\cite{tadmor2003entropy} using the entropy-conservative numerical flux
of \cite[Theorem~7.8]{ranocha2018thesis} for the compressible Euler equations
of an ideal gas in one space dimension. The initial condition
\begin{equation}
  \rho_0(x) = 1 + \frac{1}{2} \sin(\pi x),
  \quad
  v_0(x) = 1,
  \quad
  p_0(x) = 1,
\end{equation}
where $\rho$ is the density, $v$ the velocity, and $p$ the pressure,
results in a smooth and entropy-conservative solution in the periodic
domain $[0, 2]$. Integrating the entropy-conservative semidiscretization
with $N = 100$ grid nodes in time with SSP($4, 3$), where the starting
values have been obtained with the relaxation version of the classical
third-order, three-stage SSP Runge--Kutta method, yields the results
shown in Figure~\ref{fig:euler}.
Clearly, the baseline scheme is not entropy-conservative while the projection
method does not conserve the total mass. In contrast, the relaxation method
conserves both functionals.

We choose this test problem because it is a non-stiff nonlinear PDE problem
where the preservation of linear invariants is particularly interesting.
We choose SSP methods since these are often applied to computational fluid
dynamics; results for other time integration methods are similar since the
SSP property is not crucial here.

\begin{figure}[htp]
\centering
  \begin{subfigure}{0.55\textwidth}
  \centering
    \includegraphics[width=\textwidth]{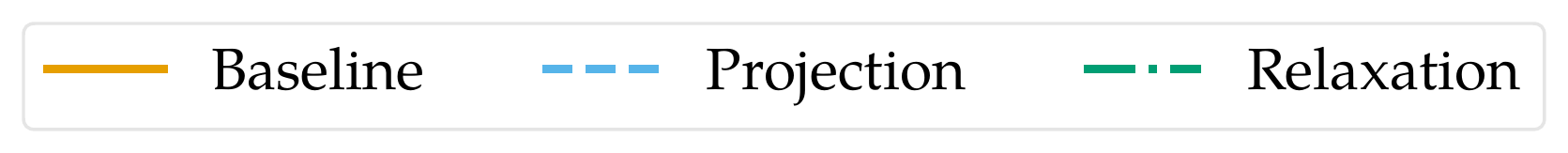}
  \end{subfigure}%
  \\
  \begin{subfigure}{0.49\textwidth}
  \centering
    \includegraphics[width=\textwidth]{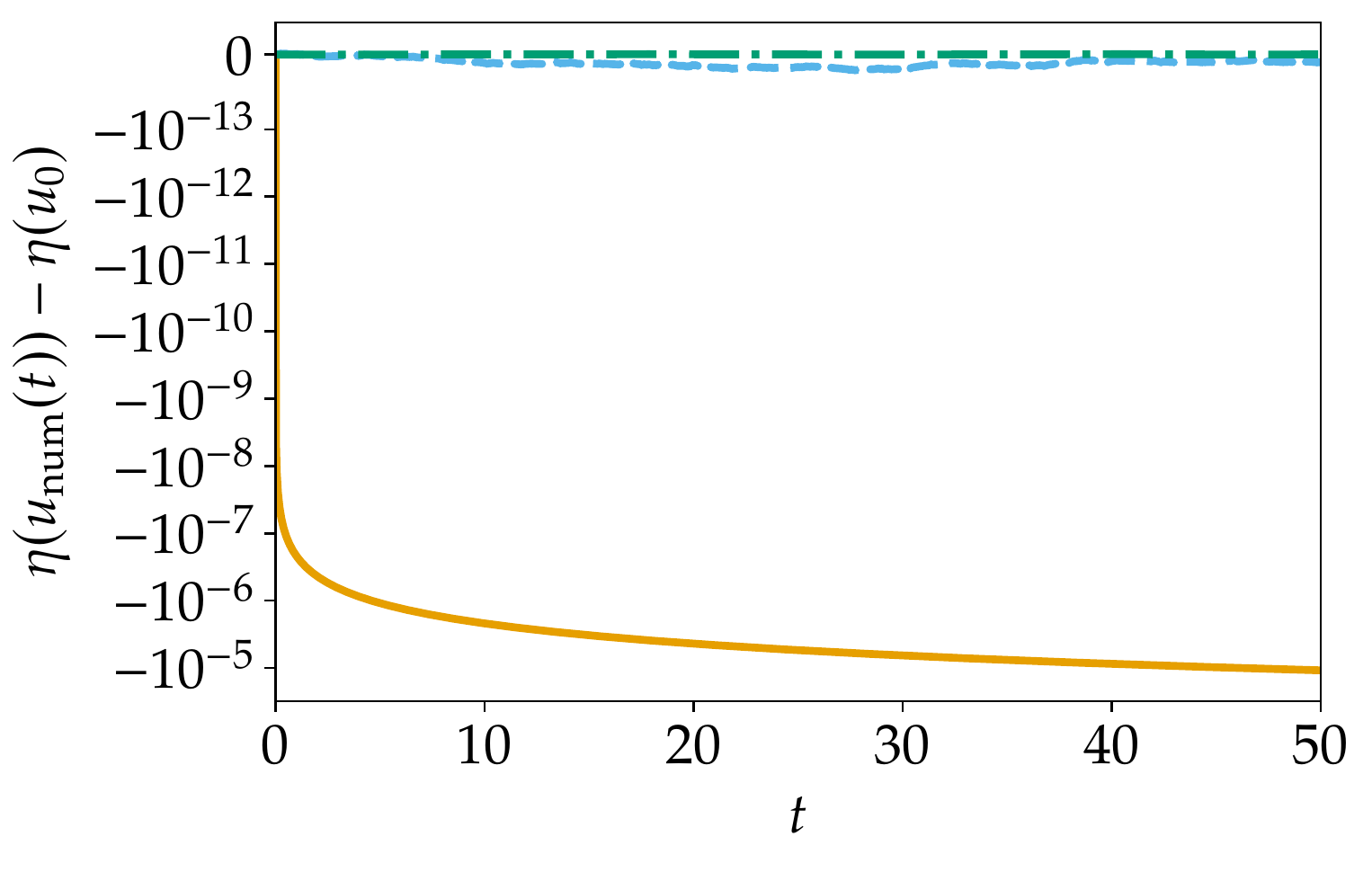}
    \caption{Total entropy change.}
  \end{subfigure}%
  \hfill
  \begin{subfigure}{0.49\textwidth}
  \centering
    \includegraphics[width=\textwidth]{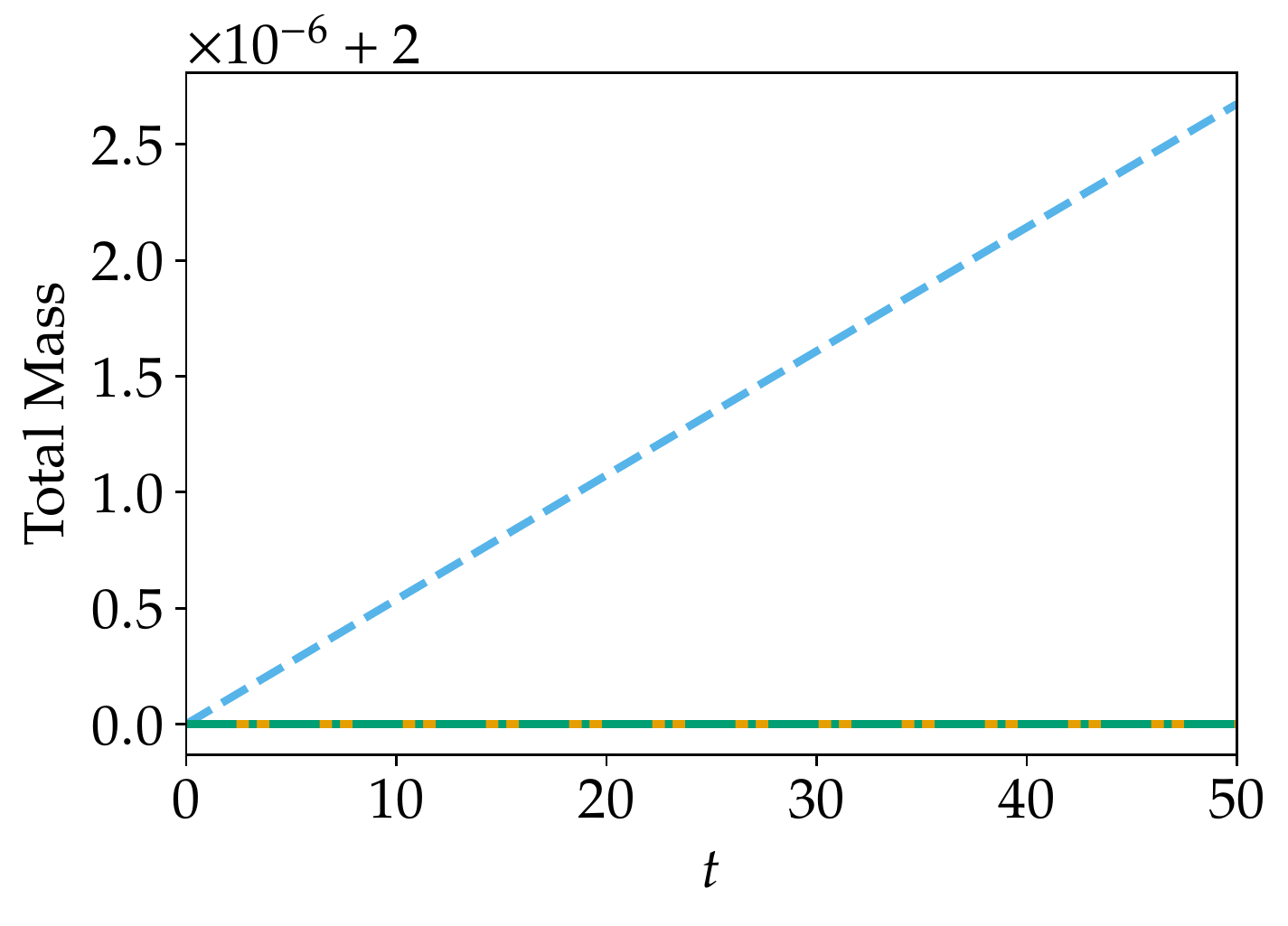}
    \caption{Total mass.}
  \end{subfigure}%
  \caption{Numerical solutions of the compressible Euler equations
           with final time $t = 50$ and time step $\dt = 0.1 \Delta x$
           using entropy-conservative finite differences and
           projection/relaxation versions of SSP($4, 3$).}
  \label{fig:euler}
\end{figure}

\subsection{Burgers' Equation}
\label{sec:burgers}

Solutions of Burgers' equation
\begin{equation}
\label{eq:burgers}
  \partial_t u(t,x) + \partial_x \frac{u(t,x)^2}{2} = 0,
  \quad
  u(0,x) = \exp(-30 x^2),
\end{equation}
in the periodic domain $[-1, 1]$ develop shocks in a finite time.
Hence, energy-conservative methods are not appropriate. Here, we
apply the same energy-dissipative semidiscretization used in
\cite{ketcheson2019relaxation} in the context of relaxation
Runge--Kutta methods, which can be written as
\begin{equation}
  \od{}{t} u_i(t)
  =
  - \frac{\fnum(u_{i}, u_{i+1}) - \fnum(u_{i-1}, u_{i})}{\Delta x}.
\end{equation}
The energy-dissipative numerical flux is obtained by adding
some dissipation to the energy-conservative flux, resulting in
\begin{equation}
  \fnum(u_-, u_+)
  =
  \frac{u_-^2 + u_- u_+ + u_+^2}{6}
  - \epsilon (u_+ - u_-).
\end{equation}
These semidiscretizations are integrated in time with SSP($3, 2$)
and SSP($4, 3$), where the starting values have been obtained
with the relaxation version of the classical third order,
three-stage SSP Runge--Kutta method.

We choose this test problem because it is a non-stiff nonlinear PDE problem
with decaying energy to demonstrate improved qualitative behavior also in
this context. We choose SSP methods again since these are often applied to
computational fluid dynamics problems.

The changes of the total energy and mass of the numerical solutions
and a semidiscrete reference solution are visualized in
Figure~\ref{fig:burgers_EC_diss}. The baseline schemes are either
anti-dissipative (for SSPRK($3, 2$)) or too dissipative
(for SSPRK($4, 3$)) compared to the reference solution, similarly
to results for Runge--Kutta methods shown in
\cite{ketcheson2019relaxation}. In contrast, the energy dissipation
of both the relaxation and the projection versions agrees very well
with the reference solution. However, the projection schemes change
the total mass while the relaxation methods conserve this invariant
of the PDE \eqref{eq:burgers}.

\begin{figure}[htp]
\centering
  \begin{subfigure}{\textwidth}
  \centering
    \includegraphics[width=\textwidth]{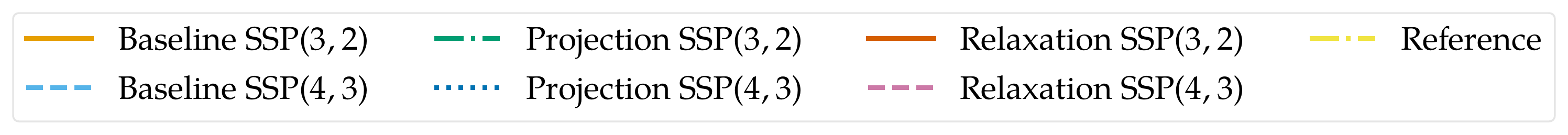}
  \end{subfigure}%
  \\
  \begin{subfigure}{0.49\textwidth}
  \centering
    \includegraphics[width=\textwidth]{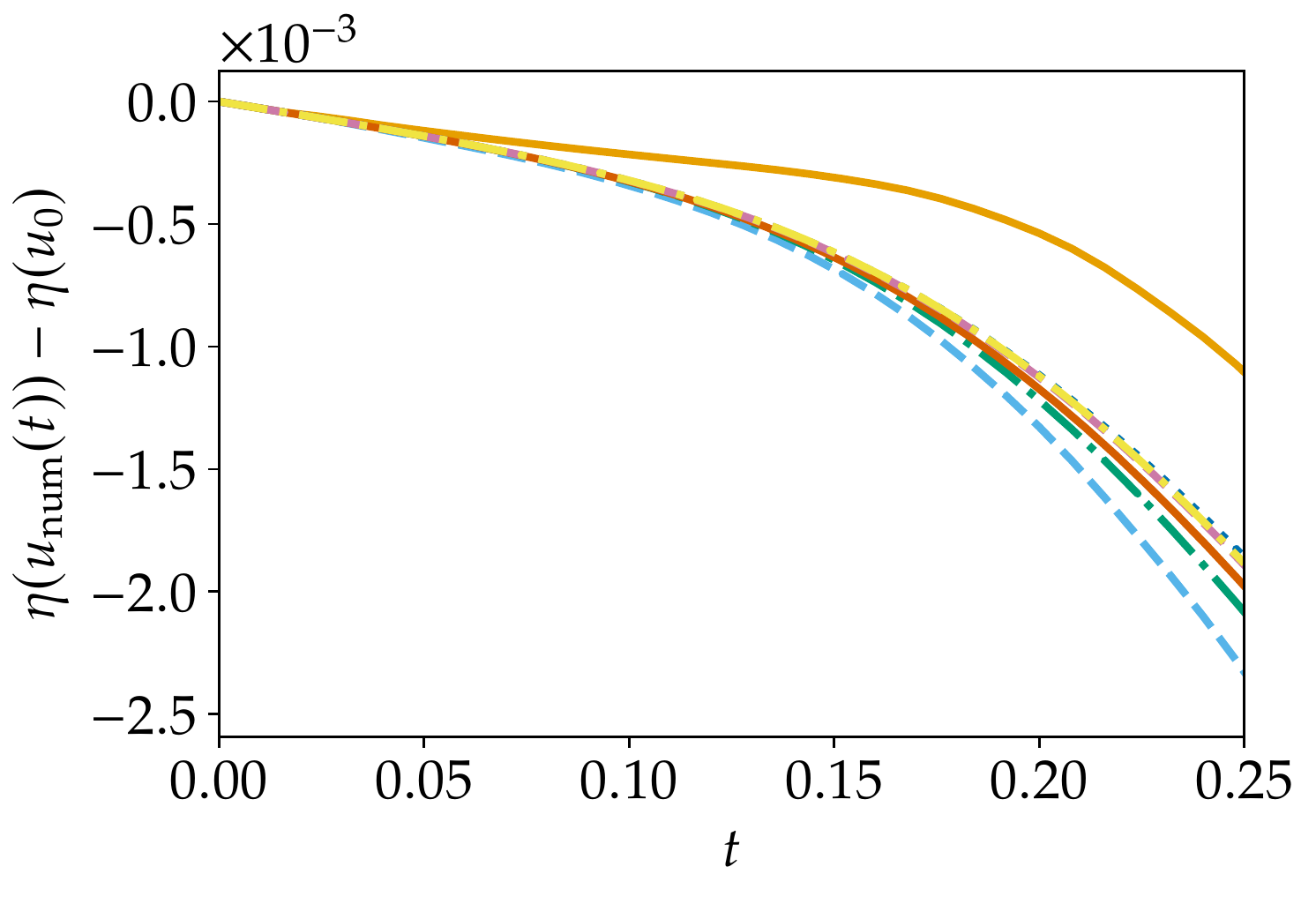}
    \caption{Total energy change.}
  \end{subfigure}%
  \hfill
  \begin{subfigure}{0.49\textwidth}
  \centering
    \includegraphics[width=\textwidth]{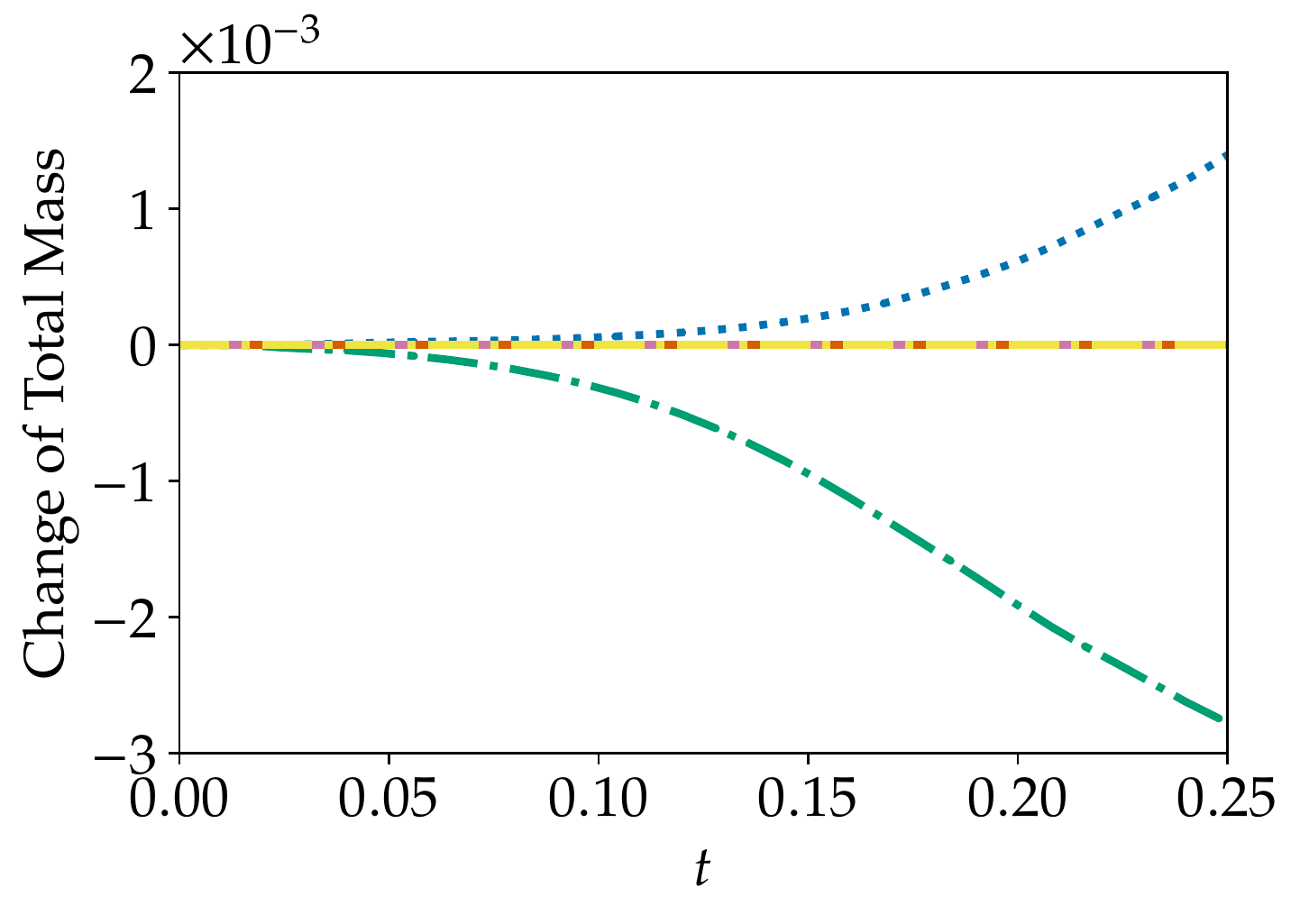}
    \caption{Total mass change.}
  \end{subfigure}%
  \caption{Numerical solutions of Burgers' equation with
           final time $t = 0.25$ and time step $\dt = 0.2 \Delta x$
           using energy-dissipative finite differences and
           projection/relaxation versions of SSP($3, 2$) and SSP($4, 3$).}
  \label{fig:burgers_EC_diss}
\end{figure}

Adding instead some dissipation to a semidiscretization based on the
central numerical flux
\begin{equation}
  \fnum(u_-, u_+)
  =
  \frac{u_-^2 + u_+^2}{4}
  - \epsilon (u_+ - u_-)
\end{equation}
does not yield a provably energy-dissipative semidiscretization in
general. However, the relaxation methods still improve the energy
evolution as visualized in Figure~\ref{fig:burgers_central_diss}.

\begin{figure}[htp]
\centering
  \begin{subfigure}{\textwidth}
  \centering
    \includegraphics[width=\textwidth]{figures/burgers_legend}
  \end{subfigure}%
  \\
  \begin{subfigure}{0.49\textwidth}
  \centering
    \includegraphics[width=\textwidth]{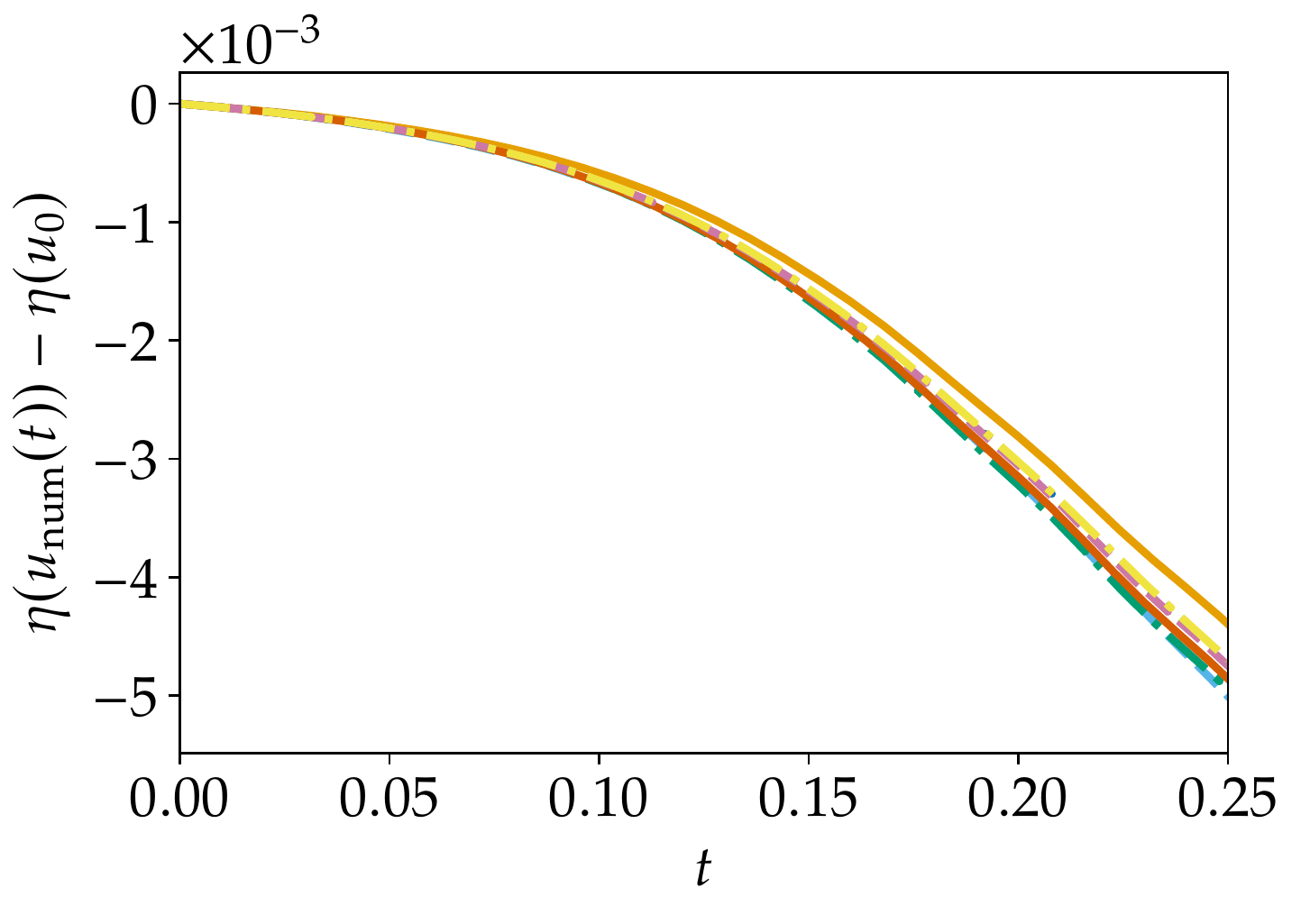}
    \caption{Total energy change.}
  \end{subfigure}%
  \hfill
  \begin{subfigure}{0.49\textwidth}
  \centering
    \includegraphics[width=\textwidth]{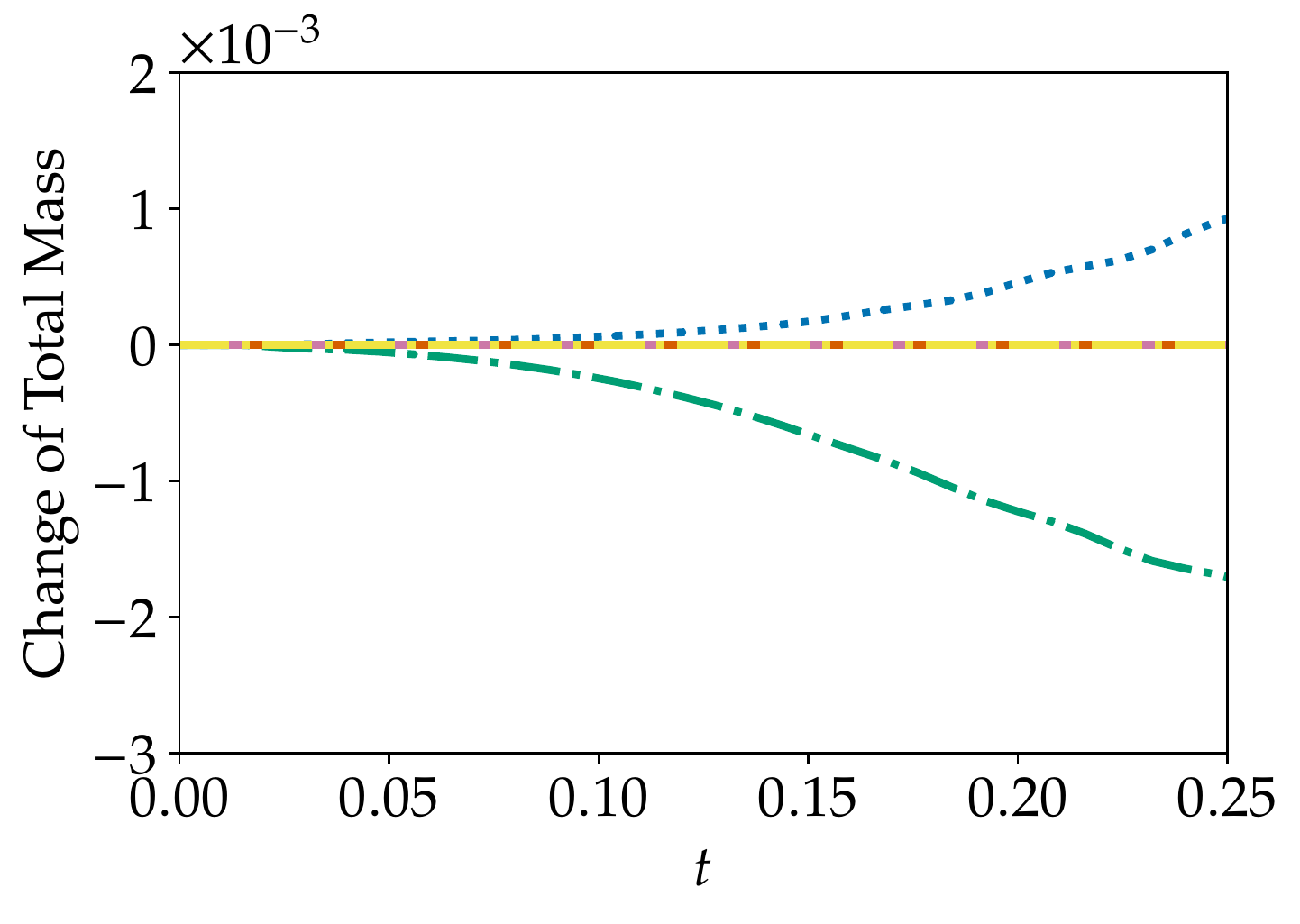}
    \caption{Total mass change.}
  \end{subfigure}%
  \caption{Numerical solutions of Burgers' equation with
           final time $t = 0.25$ and time step $\dt = 0.2 \Delta x$
           using central finite differences with dissipation and
           projection/relaxation versions of SSP($3, 2$) and SSP($4, 3$).}
  \label{fig:burgers_central_diss}
\end{figure}

\subsection{Linear Advection with Inflow}
\label{sec:linear_advection}

Solutions of the linear advection equation
\begin{equation}
\label{eq:linear_advection}
\begin{aligned}
  \partial_t u(t,x) + \partial_x u(t,x) &= 0,
  && t \in (0, 6),\, x \in (0, 3),
  \\
  u(0,x) &= 0,
  && x \in [0, 3],
  \\
  u(t, 0) &= \sin(\pi t),
  && t \in [0, 6],
\end{aligned}
\end{equation}
do neither conserve nor dissipate the energy $\frac{1}{2} \norm{u}_{L^2}^2$
because of the boundary condition. Instead, the energy of the analytical
solution increases till $t = 3$ to its final value $0.75$ and stays
constant thereafter.
We choose this test problem because it results in a non-monotone behavior of
the energy. We choose SSP methods again since these are often applied to
computational fluid dynamics problems.

Using the classical second-order summation-by-parts operator with
simultaneous approximation terms to impose the boundary condition weakly
\cite{svard2014review} on $N = 200$ uniformly spaced nodes yields an ODE
with a similar behavior. As visualized in Figure~\ref{fig:linear_advection},
the baseline SSP($3, 2$) method results in an increase of the energy that
is slightly bigger than that of the reference solution obtained by
SSP($4, 3$) with much smaller time steps. Instead, the energy variation
enforced by the projection and relaxation methods is visually
indistinguishable from the reference value. The slight variations of the
energy for $t > 3$ are caused by the spatial semidiscretizations using a
weak imposition of the boundary condition.

This example demonstrates that relaxation methods can also be useful for
non-dissipative systems where energy or entropy estimates can still
be obtained and are of interest. In \cite{ranocha2020relaxation}, a similar
lid-driven cavity flow with a heated wall for the Navier--Stokes equations
has been solved with relaxation Runge--Kutta methods.

\begin{figure}
\centering
  \includegraphics[width=0.8\textwidth]{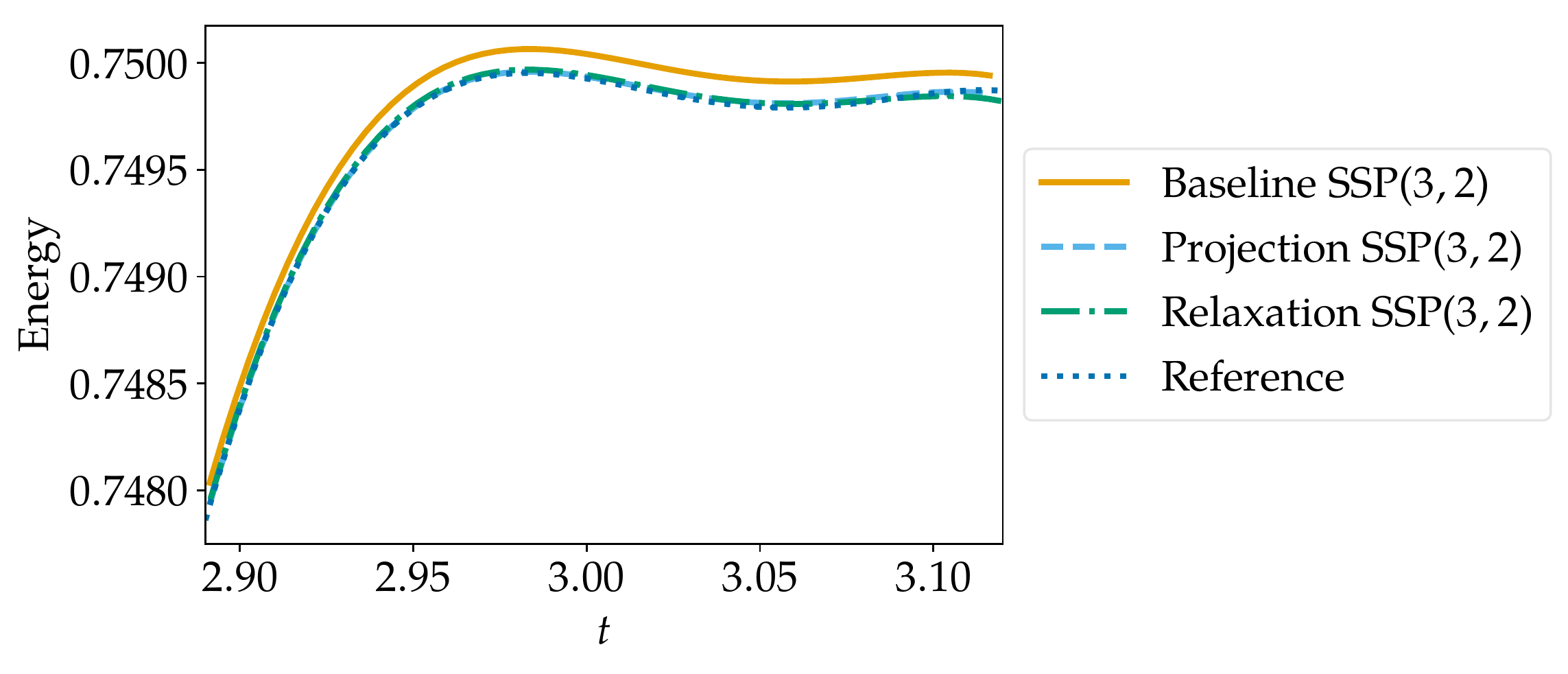}
  \caption{Energy of numerical solutions of the linear advection equation
           \eqref{eq:linear_advection} computed with or without
           projection/relaxation methods and SSP(3, 2).}
  \label{fig:linear_advection}
\end{figure}

\subsection{Some Remarks on the Costs of Relaxation}
\label{sec:costs}

For a quadratic or cubic entropy functional, the relaxation parameter can
be computed explicitly.  For more general entropies, it must be found via
numerical iteration.

We have used SciPy \cite{virtanen2019scipy} to solve the scalar equations for
the relaxation parameter $\gamma^n$.  Our implementations are based on functions
written in pure Python and are not adapted to the specific problems. A detailed
assessment of the computational cost of solving for the relaxation parameter
$\gamma^n$ requires a more refined implementation and is beyond the scope of
this work. Nevertheless, we give some preliminary discussion of the costs here,
emphasizing that these numbers should be viewed as very crude upper bounds on
the cost associated to relaxation.

For explicit time-stepping methods and very inexpensive right-hand sides, the
costs of naive implementations of the relaxation approach are significant.
For the example of the compressible Euler equations in Section~\ref{sec:euler},
relaxation methods increase the runtime by a factor between two and three.
For Burgers' equation as in Section~\ref{sec:burgers}, the total runtime
increases by less than a factor of two.
Although the relaxation technique increases the runtime significantly in this
naive implementation, it is still cheaper than decreasing the time step to get
basically the same results.

Relaxation methods were applied to large-scale PDE discretizations of
the compressible Euler and Navier-Stokes equations in three space dimensions in
\cite{ranocha2020relaxation,ranocha2020fully}. In that context, the cost
associated with solving the single scalar equation for the relaxation parameter
$\gamma^n$ is much less than the cost of evaluating
the right-hand side of the ODE during one time step.

When implicit time-stepping is used, the cost of relaxation is less significant.
This was already discussed in \cite{ranocha2020relaxationHamiltonian} for the
KdV equation. There, the energy-conserving methods decreased the runtime despite
the added overhead of computing the relaxation or projection step. The reason
for this is the improved accuracy and hence decreased costs to solve the
nonlinear stage equations. The relaxation method benefited additionally from
taking larger effective time steps since the relaxation parameter $\gamma^n > 1$.

\section{Summary and Conclusions}
\label{sec:summary}

We have extended the framework of relaxation methods for the numerical
solution of initial-value problems from Runge--Kutta methods to general
time integration schemes with order of accuracy $p \geq 2$.
By solving a single scalar algebraic equation per time step, the evolution
in time of a given functional can be preserved. This includes
functionals that are conserved or dissipated, as well as others
for which estimates of the time evolution are available.
For convex functionals, additional insights such as the possibility to
add dissipation in time have been provided.

For certain classes of relaxation linear multistep methods, high-order
accuracy is still attained even if a fixed-coefficient
method is used. Nevertheless, we recommend the use of methods that
correctly account for the step size variation, since such methods
gave overall better results in numerical tests.

In contrast to orthogonal projection methods, relaxation methods preserve
all linear invariants that are preserved by the baseline time integration
scheme (which are all linear invariants of the ODE for general linear
methods). This property can be very important, e.g.\ for conservation laws.

We have also studied the impact of the relaxation approach on other
stability properties of time integration methods. In particular, zero
stability and strong stability preserving properties of linear multistep
and Runge--Kutta methods are not changed significantly.

While relaxation methods appear to provide good results in our
numerical experiments, further practical experience on a wide
range of problems is still needed to determine their general
effectiveness. Other areas of ongoing research include
the development of other means to estimate
the change of a dissipated functional $\eta$ and development of a spatially
localized relaxation approach for conservation laws with convex entropies.

\appendix
\section*{Acknowledgments}

Research reported in this publication was supported by the
King Abdullah University of Science and Technology (KAUST).
The project ``Application-domain-specific highly reliable IT solutions'' has
been implemented with the support provided from the National Research,
Development and Innovation Fund of Hungary, and financed under the scheme
Thematic Excellence Programme no. 2020-4.1.1-TKP2020
(National Challenges Subprogramme).

\section{Superconvergence for Euclidean Hamiltonian Problems}
\label{sec:superconvergence}

Here, we present a generalization of the superconvergence Theorem~4.1
of \cite{ranocha2020relaxationHamiltonian} to general energy-conservative
B-series \cite{hairer1974butcher} time integration methods, see e.g.\
\cite{mclachlan2016b} and the references cited therein. This general class of
time integration schemes includes among others Runge--Kutta methods,
linear multistep methods, Taylor series methods, and Rosenbrock methods.

Consider a Hamiltonian $H$ that is a smooth function of the squared
Euclidean norm, i.e.
\begin{equation}
  H(q,p) = G\bigl( (|q|^2 + |p|^2) / 2 \bigr),
\end{equation}
where $G$ is a smooth function. The corresponding Hamiltonian system is
\begin{equation}
\label{eq:nonlinear-Euclidean-Hamiltonian}
  \begin{cases}
    u'(t) = f(u(t)), \\
    u(0) = u^{0},
  \end{cases}
  \quad
  u(t) = \begin{pmatrix} q(t) \\ p(t) \end{pmatrix},
  \quad
  f(q,p) = g\bigl( (|q|^2 + |p|^2) / 2 \bigr) \begin{pmatrix} p \\ -q \end{pmatrix},
\end{equation}
where $g = G'$. We refer to \eqref{eq:nonlinear-Euclidean-Hamiltonian} as a
\emph{Euclidean Hamiltonian} problem.
For this class of problems, nominally odd-order energy-conservative
B-series time integration methods are superconvergent.

The proof relies on a special geometric structure of the error. As shown
in \cite[Lemma~A.4]{ranocha2020relaxationHamiltonian}, for \eqref{eq:nonlinear-Euclidean-Hamiltonian}
the local error can be divided into two parts. The part that is along the
manifold of constant $H$ includes all terms with even powers of $\dt$, while
the part orthogonal to the manifold of constant $H$ includes all terms with odd
powers of $\dt$. Thus a solution that preserves $H$ has no error terms with
odd powers of $\dt$.
\begin{theorem}
\label{thm:nonlinear-Euclidean-Hamiltonian}
  Consider the general nonlinear Euclidean Hamiltonian system
  \eqref{eq:nonlinear-Euclidean-Hamiltonian}
  with $g \neq 0$ and a B-series time integration method of order $p$
  with or without applying orthogonal projection or relaxation.

  If the method conserves the Hamiltonian, the expansion of the local
  error contains only odd powers of $\dt$. In particular, its order of
  accuracy for the system \eqref{eq:nonlinear-Euclidean-Hamiltonian} is
  $p+1$ if $p$ is odd.
\end{theorem}
\begin{proof}
  It suffices to consider the local error after one step,
  which can be written as
  \begin{equation}
    u^{n} - u(t^{n-1}+\dt)
    =
    \sum_{k = p+1}^\infty \dt^k \sum_{\abs{t} = k} e_{t} F(t)(u^{n-1}),
  \end{equation}
  where $e_t$ are some coefficients of the local error, the second
  sum is a sum over all rooted trees~$t$ of order~$k$,
  cf.\ \cite[Chapter~3]{butcher2016numerical}, and $F(t)(u^{n-1})$
  is an elementary differential of $f$ evaluated at $u^{n-1}$.

  Since the method is energy-conservative, it conserves the Euclidean norm.
  Hence, for any $k \geq p+1$,
  \begin{equation}
  \begin{aligned}
    0
    &=
    \frac{\mathrm{d}^{k}}{\dif \dt^{k}}\bigg|_{\dt = 0} \left(
      \| u^{n} \|^2 - \| u(t^{n-1}+\dt) \|^2
    \right)
    \\
    &=
    \frac{\mathrm{d}^{k}}{\dif \dt^{k}}\bigg|_{\dt = 0}
    \scp{ u^{n} - u(t^{n-1}+\dt) }{ u^{n} + u(t^{n-1}+\dt) }
    \\
    &=
    \scp{ \frac{\mathrm{d}^{k}}{\dif \dt^{k}}\bigg|_{\dt = 0} \left( u^{n} - u(t^{n-1}+\dt) \right) }{ 2 u^{n-1} }
    \\
    &=
    2k \sum_{\abs{t} = k} e_{t} \scp{ F(t)(u^{n-1}) }{ u^{n-1} }.
  \end{aligned}
  \end{equation}
  For even $k$, $F(t)(u^{n-1}) \parallel u^{n-1}$ for $\abs{t} = k$
  because of \cite[Lemma~A.4]{ranocha2020relaxationHamiltonian}.
  Hence, $\sum_{\abs{t} = k} e_{t} F(t)(u^{n-1}) = 0$ for even $k$.
\end{proof}

\section{An Example for Which Relaxation LMMs are Exact}

Here we consider a problem from \cite{ranocha2020relaxation}:
\begin{subequations} \label{conserved-exponential}
\begin{align}
    u_1'(t) & = -\exp(u_2), \\
    u_2'(t) & = \exp(u_1).
\end{align}
\end{subequations}
The entropy $\eta(u) = \exp(u_1)+\exp(u_2)$ is conserved.
The exact solution is given by
\begin{align}
    u_1(t) & = \log\left(\frac{C\eta}{C+e^{\eta t}} \right), &
    u_2(t) & = \log\left(\frac{\eta e^{\eta t}}{C+e^{\eta t}} \right),
\end{align}
where $\eta = \exp(u_1)+\exp(u_2)$ and $C = e^{u_1(0)-u_2(0)}.$
\begin{theorem}
Let a relaxation LMM of order two or greater be applied to
\eqref{conserved-exponential}, with $\dt$ small enough such that there exists
a value of $\gamma$ satisfying \eqref{eq:eta-est-conservative}. If the
starting values are exact, then
the numerical solution is (in the absence of rounding errors) exact at each step.
\end{theorem}
\begin{proof}
First observe that the equations
\begin{equation}
  w = u_2-u_1,
  \qquad
  \eta = \exp(u_1) + \exp(u_2),
\end{equation}
define a bijection between $(u_1,u_2)\in \R^2$ and $(w,\eta)\in \R\times\R^+$.
Thus \eqref{conserved-exponential} is equivalent to the system
\begin{subequations} \label{easy-ODE}
\begin{align}
    w'(t) & = \eta, \\
    \eta'(t) & = 0,
\end{align}
\end{subequations}
with the solution $w(t)=\eta(0) t + w(0)$ and $\eta(t)=\eta(0)$. We will show
that any relaxation LMM integrates the system \eqref{easy-ODE} exactly.
We can write any LMM as
\begin{equation}
  \unew
  =
  \sum_{j=0}^{k-1}\alpha_j u^{n-k+j} + \dt \sum_{j=0}^k \beta_j f(u^{n-k+j}).
\end{equation}
Taking the difference of the formulas for $u^n_2$ and $u^n_1$ we obtain
\begin{equation}
w^n = \sum_{j=0}^{k-1}\alpha_j w^{n-k+j} + \dt \eta \sum_{j=0}^k \beta_j.
\end{equation}
Since the starting values are exact, $w^{n-k+j} = w(t^{n-k+j})$.
Due to the consistency of the LMM, this means that the formula above is
exact; i.e. $w^n = w(t^{n-1}+\dt)$. Next we take
\begin{equation}
w^n_\gamma = w^{n-m} + \gamma(w^n - w^{n-m}) = w(t^n+\gamma\dt),
\end{equation}
since $w(t)$ is linear. Thus the numerical solution given by the relaxation
LMM gives the exact value for $w$. It also gives the exact value for $\eta$,
since this is precisely what is enforced by relaxation.
\end{proof}
\begin{remark}
Note that without the relaxation step, the solution is not exact since
the equation for $w(t)$ is solved exactly but that for $\eta(t)$ is not.
Projection methods do not solve \eqref{conserved-exponential} exactly, because the
projection step does not preserve the exact solution of the linear ODE for $w(t)$.
Relaxation Runge--Kutta methods (and other multistage methods) also do not
solve \eqref{conserved-exponential} exactly, because they involve iterated
evaluation of the nonlinear RHS within a single step, so they do not effectively
integrate the linear ODE for $w(t)$ exactly.
\end{remark}

\printbibliography

\end{document}